\documentclass[11pt]{amsart}

\usepackage[colorlinks=true, pdfstartview=FitV, linkcolor=blue, 
citecolor=blue]{hyperref}

\usepackage{amssymb,amsmath}
\usepackage{bbm}
\usepackage{a4wide}

\DeclareMathAlphabet{\mymathbb}{U}{bbold}{m}{n}

\newcommand{\rev}{\color{blue}}

\newtheorem{theorem}{Theorem}[section]
\newtheorem{prop}[theorem]{Proposition}
\newtheorem{lemma}[theorem]{Lemma}     
\newtheorem{fact}[theorem]{Fact}
\newtheorem{coro}[theorem]{Corollary}
\theoremstyle{definition}
\newtheorem{definition}[theorem]{Definition}
\newtheorem{example}[theorem]{Example}
\newtheorem{remark}[theorem]{Remark}

\newcommand{\ts}{\hspace{0.5pt}}
\newcommand{\nts}{\hspace{-0.5pt}}
\newcommand{\RR}{\mathbb{R}\ts}
\newcommand{\CC}{\mathbb{C}}
\newcommand{\JJ}{\mathbb{J}}
\newcommand{\ZZ}{\mathbb{Z}}
\newcommand{\NN}{\mathbb{N}}

\newcommand{\cA}{\mathcal{A}}
\newcommand{\cC}{\mathcal{C}}

\newcommand{\cM}{\mathcal{M}}

\newcommand{\emb}{\cM^{\ts\vphantom{g}\mathrm{e}}_{d}}
\newcommand{\gemb}{\cM^{\mathrm{ge}}_{d}}

\newcommand{\ma}{\mathfrak{m}_{\mathsf{a}}}
\newcommand{\mg}{\mathfrak{m}_{\mathsf{g}}}

\newcommand{\ee}{\ts\mathrm{e}}
\newcommand{\dd}{\,\mathrm{d}\ts}
\newcommand{\ii}{\ts\mathrm{i}}
\newcommand{\bs}{\boldsymbol}
\newcommand{\one}{\mymathbb{1}}
\newcommand{\nix}{\mymathbb{0}}
\newcommand{\Mat}{\mathrm{Mat}}
\newcommand{\GL}{\mathrm{GL}}
\newcommand{\cent}{\mathrm{cent}}
\newcommand{\comm}{\mathrm{comm}}
\newcommand{\imag}{\mathrm{Im\ts}}

\newcommand{\tr}{\mathrm{tr}}

\newcommand{\diag}{\mathrm{diag}}

\newcommand{\exend}{\hfill$\Diamond$}
\newcommand{\defeq}{\mathrel{\mathop:}=}
\newcommand{\eqdef}{=\mathrel{\mathop:}}

\newcommand{\myfrac}[2]{\frac{\raisebox{-2pt}{$#1$}}
  {\raisebox{0.5pt}{$#2$}}}
\newcommand{\cc}[1]{\ts \overline{\nts #1 \nts} \ts }
\newcommand{\mubar}{\ts\ts \overline{\nts\nts \mu \nts} \ts }

\begin{document}

\title[Embedding of Markov matrices for $d\leqslant 4$]
{Embedding of Markov matrices for $\bs{d \leqslant 4}$}

\author{Michael Baake}
\address{Fakult\"at f\"ur Mathematik, Universit\"at Bielefeld, \newline
       \indent  Postfach 100131, 33501 Bielefeld, Germany}

\author{Jeremy Sumner}
\address{School of Natural Sciences, Discipline of Mathematics,
         University of Tasmania,
    \newline \indent Private Bag 37, Hobart, TAS 7001, Australia}

\begin{abstract} 
  The embedding problem of Markov matrices in Markov semigroups is a
  classic problem that regained a lot of impetus and activities
  through recent needs in phylogeny and population genetics.  Here, we
  give an account for dimensions $d\leqslant 4$, including a complete
  and simplified treatment of the case $d=3$, and derive the results
  in a systematic fashion, with an eye on the potential applications.
  
  Further, we reconsider the setup of the corresponding problem for
  time-inhomogeneous Markov chains, which is needed for real-world
  applications because transition rates need not be constant over
  time.  Additional cases of this more general embedding occur for any
  $d\geqslant 3$. We review the known case of $d=3$ and describe the
  setting for future work on $d=4$.
\end{abstract}


\maketitle

\bigskip

\section{Introduction}

Markov models are the backbone of many if not most stochastic
processes.  In real-world applications, they usually appear with
finite state spaces, and come in two flavours, namely with discrete or
with continuous time. While discrete-time Markov chains often appear
simpler on first inspection, they are connected to their
continuous-time counterparts via the obvious question when they admit
a continuous interpolation. This was identified as an important
problem by Elfving \cite{Elfving} because pure discrete-time models
may have some problems or even inconsistencies for the intended
application. Clearly, this question is also relevant in biological
models, for instance in phylogeny and population genetics. Indeed,
when modelling the time evolution of genetic sequences built from the
nucleotide alphabet $\{ A, G, C, T \}$, one is right in the middle of
this type of interpolation problem. It can essentially be rephrased as
asking whether a given Markov matrix can be embedded into a Markov
semigroup or, more generally, into a Markov flow, thus referring to
underlying continuous-time processes that are homogeneous or
non-homogeneous in time, respectively. The purpose of this paper is to
provide methods and results to answer the embedding problem for this
type of application. \vspace*{1mm}

Let us be a bit more specific.  Phylogenetic models typically make the
strong assumption that random changes in nucleotide sequences follow a
Markov process that is both stationary (in nucleotide composition) and
homogeneous (meaning that instantaneous substitution rates are
constant in time and across biological lineages). However, there is
abundant biological evidence that the molecular evolution of certain
groupings of organisms does not strictly adhere to the typical
assumption of a stationary and homogeneous Markov process; for
example, Cox et al.~\cite{Cox} establish significant nucleotide
compositional differences across the tree of life, while Jayaswal et
al.~\cite{Jay2} establish the same for a yeast phylogeny. For this
reason, there are various studies that assess the possibility of
fitting non-stationary models of molecular evolution; compare
\cite{Jay}.  The embedding problem therefore has important practical
application in phylogenetics as its solutions allow (at least
theoretically) for the detection of non-homogeneous Markov processes
in historical molecular evolution. \vspace*{1mm}

Let us now describe our setting from a more mathematical perspective.
A \emph{Markov matrix} $M$ is a non-negative square matrix with all
row sums equal to one, which is the convention we use here. The set of
all Markov matrices with $d$ states (or in $d$ dimensions) is denoted
by $\cM_d$, which is a closed convex subset of the real
$d {\times} d$-matrices.  A \emph{rate matrix} $Q$ is a real matrix
with non-negative off-diagonal entries and all row sums equal to
zero. Such matrices are also known as \emph{Markov generators}, or
simply \emph{generators}, due to their connection with
(time-homogeneous) Markov semigroups, as given by
$\{ \ee^{t \ts Q} : t \geqslant 0 \}$, which is a monoid of Markov
matrices.

It is an old question, known as Elfving's interpolation problem
\cite{Elfving}, whether a given Markov matrix $M$ can occur within a
Markov semigroup. This is equivalent to asking whether there is a
generator $Q$ such that $M=\ee^Q$, which simply sets a particular time
scale. We call this the (classic) \emph{embedding problem}. It became
famous through a foundational paper by Kingman \cite{King}, which also
included a simple criterion for $d=2$ due to Kendall and attracted a
lot of research for many years; see \cite{Davies, BS} and references
therein. Several general characterisations were found, while concrete,
applicable criteria in higher dimensions turned out to be more
difficult. After settling $d=3$, see \cite{Davies,BS} and references
therein for an account of the history, and several papers in
mathematical finance, sociology and statistics, compare
\cite[Sec.~2.3]{Higham}, the interest in the problem faded somewhat,
as no driving application was in sight. In particular, no systematic
treatment of $d>3$ was done. A similar fate was met by the generalised
embedding problem of a given Markov matrix in a time-inhomogeneous
process \cite{FS,JR}.  Also, it was clear that, in higher dimensions,
the embedding problem will become increasingly more complex, due to
topological properties already discussed in \cite{King} and due to the
multiple possibilities for degenerate spectra with repeated Jordan
blocks.  To some extent, this starts in $d=4$, and no impetus was
visible to work on a classification. \vspace*{1mm}

This changed through the rise of bioinformatics, which needed the case
$d=4$ solved explicitly, with effective and concrete criteria, due to
the importance of various Markov processes on the genetic (or
nucleotide) alphabet $\{ A, G, C, T \}$.  Here, processes of molecular
evolution became relevant in phylogenetics, as one typical problem is
the inference of continuous-time processes such as mutation or
recombination from discrete data sets.  The transition rates of
real-world Markov processes should be expected to change over time
\cite{Cox}. For this and various other reasons, discrete-time models
that are not embeddable in any reasonable way appear, in our opinion,
questionable or too restrictive for applications to overlapping
generations or to other processes that require continuous time. In
particular, this is certainly a natural point of view for processes
over a long time period, as one encounters in phylogenetics. Recently,
the embedding problem for $d=4$ was largely solved in \cite{CFR} via a
detailed case distinction on the basis of the possible Jordan normal
forms, including algorithms to decide some of the more complicated
cases. One small disadvantage of this approach is some dependence on
the similarity transform between the matrix and its \emph{Jordan
  normal form} (JNF). \vspace*{1mm}

Here, we take a different, more algebraic path to the problem for
$d\leqslant 4$ that is based on the degree of the minimal polynomial
of the matrix.  While some cases will still require the JNF, the
treatment of cyclic matrices (which are those where minimal and
characteristic polynomial agree and are thus generic in a
measure-theoretic sense --- to be made more precise below) will become
more systematic, and the emerging embedding criteria are explicit and
easy to implement. The golden key in this generic case, as well as in
several of the other ones, is the proper use of the \emph{spectral
  mapping theorem} (SMT), see \cite[Sec.~5.3]{HP} for the general
setting and Eq.~\eqref{eq:SMT} below for a formulation, in conjunction
with Frobenius' theorem on the centraliser of a cyclic matrix, as
developed in some detail in \cite{BS}; see Theorem~\ref{thm:Frob} for
specific details. In most cases, this leads to explicitly solving a
simple set of linear equations and checking whether the solution is a
generator or not.

Then, we look at the corresponding problem with time-dependent
generators. While this does not give new cases for $d=2$, things
change already for $d=3$, as has been known since the 1970s
\cite{Good,Joh73}. Consequently, embeddability in this generalised
sense is relevant, also in view of the applications. Fortunately,
quite a bit of literature exists already, and a powerful connection
with a problem from control theory \cite{FS,JR,Fryd80a,Fryd80b,Fryd83}
makes some concrete results possible, though mainly for $d=3$.  Here,
we review and extend these results in our setting.  \smallskip

The paper is organised as follows. After some recollections of general
results in Section~\ref{sec:prelim}, which includes the complete
answer for $d=2$, we discuss the case $d=3$ in
Section~\ref{sec:three}.  While this is classic material, we are not
aware of a systematic presentation in one place, and this
simultaneously allows us to introduce further methods and tools as we
go along. In this sense, our paper is both a partial review and an
extension of known results at the same time. In particular, we take a
fresh look at the most difficult case, which consists of the
diagonalisable Markov matrices with a negative eigenvalue of
multiplicity $2$.  Here, we give a complete solution that is both
simple and constructive. A summary is provided in Table~\ref{tab:d3}.

In our exposition, for the sake of better readability, we often give
the arguments first and only then state the result, thus deviating
from a purely formal presentation. In this context, in line with many
other mathematical texts of a more expository nature, we use the
symbol $\,\Box\,$ to denote the end or the \emph{absence} of a proof,
the latter when we cite a result from another source or when the
arguments for the proof were given prior to the statement of the
result.

Section~\ref{sec:four} then deals with $d=4$, in the same spirit,
where we spend some time to spell out the cyclic cases; a summary is
given in Table~\ref{tab:d4}. This is followed by applications to some
important examples from phylogenetics in Section~\ref{sec:appl},
which, together with the progress in \cite{CFR}, motivated the entire
endeavour in the first place.  Finally, in Section~\ref{sec:inhom}, we
summarise the embedding problem for time-inhomogeneous Markov chains,
where we can restrict our main attention to piecewise continuous
families of generators without loss of generality, as a result of the
Bang-Bang principle from control theory that applies in this setting
\cite{Joh73,JR,Fryd80b}. Here, we set the scene for future work, in
particular on $d=4$.
\clearpage

\section{Preliminaries and general background
  results}\label{sec:prelim}

Let us recall some notation and results on matrices that we will
employ throughout, mainly following \cite{Gant,HJ,BS,BS2}.  We use
$\Mat (d,\RR)$ and $\Mat (d,\CC)$ for the rings of
$d{\times} d$-matrices over the real and the complex numbers,
respectively. We use $\one_d$ to denote the unit matrix, where we will
simply write $\one$ whenever the dimension is clear from the
context. Further, we use
\begin{equation}\label{eq:def-A}
    \cA^{(d)}_{\ts 0} \, \defeq \, \{ A \in \Mat (d,\RR) : 
    \text{all row sums are } \ts 0 \} \ts ,
\end{equation}
which is a non-unital algebra of dimension $d\ts (d\ts {-}1)$.  Given
some $B\in \Mat (d,\CC)$, its \emph{spectrum} is the set of
eigenvalues of $B$, denoted by $\sigma (B)$ and primarily viewed as a
set (and not as a multi-set, unless stated otherwise). The
\emph{characteristic polynomial} of $B$ is
\begin{equation}\label{eq:def-cp}
     p_{\nts_B} (x) \, \defeq \, \det (B - x \one) \, = \, (-1)^d
     \prod_{\lambda\in \sigma (B)} (x-\lambda)^{\ma (\lambda)} ,
\end{equation}
where $\ma (\lambda)$ is the \emph{algebraic} multiplicity of
$\lambda$, so $\sum_{\lambda \in \sigma (B)} \ma (\lambda) = d$, while
the \emph{geometric} multiplicity is
$\mg (\lambda) = \dim (V_{\lambda})$, with
$V_{\lambda} = \{ x \in \CC^d : B x = \lambda x \}$ denoting the
eigenspace for $\lambda\in\sigma(B)$ as usual.  When $B$ is a matrix
and $p (x)$ is any polynomial in one variable, also $p (B)$ is well
defined, and the SMT \cite[Thm.~10.33]{Rudin} applied to matrices
states that
\begin{equation}\label{eq:SMT}
  \sigma \bigl( p (B) \bigr) \, = \, p \bigl( \sigma (B) \bigr) \, \defeq \,
  \{ \ts p (\lambda) : \lambda \in \sigma (B) \} \ts ,
\end{equation} 
which holds for the spectrum including multiplicities. In this
context, the Cayley--Hamilton theorem \cite[Thm.~2.4.3.2]{HJ} asserts
that $p_{\nts_B} (B)= \nix$, where we use $\nix$ (or $\nix_d$) to
denote the all-zero matrix.

We shall also need the \emph{minimal polynomial} of $B$, called
$q_{_B}$, which is the minimal monic factor of $p_{\nts_B}$ that still
satisfies $q_{_B} (B) = \nix$. It is given by
\begin{equation}\label{eq:def-mp}
    q_{_B} (x) \, = \prod_{\lambda\in\sigma (B)}
    (x-\lambda)^{r^{}_{\max} (\lambda) } ,
\end{equation}
where $r^{}_{\max}(\lambda)$ is the largest dimension of the elementary 
Jordan blocks for the eigenvalue $\lambda$. Note that we have dropped 
the factor $(-1)^d$ to make it monic (leading coefficient is $1$), which is 
common, but not always done.

A matrix $B\in\Mat (d, \CC)$ is called \emph{simple} when its
eigenvalues are distinct, which is to say that
$\ma (\lambda) = \mg (\lambda) = 1$ for all $\lambda \in \sigma (B)$,
and \emph{degenerate} otherwise. Further, $B$ is called \emph{cyclic}
when $q_{_B} = (-1)^d \, p_{\nts_B}$, that is, when characteristic and
minimal polynomial agree (possibly up to a sign). Cyclic matrices are
the ones with $\mg (\lambda) = 1$ for all $\lambda\in\sigma (B)$,
which means that the corresponding Jordan blocks in the standard
(complex) JNF of $B$ are $\JJ_{\ma (\lambda)} (\lambda)$. Here,
$\JJ_n (\lambda) = \lambda \one_n + N_n$, with $N_n \in \Mat (n,\RR)$
denoting the matrix with $1$s on the first super-diagonal and $0\ts$s
everywhere else. A matrix $\JJ_n (\lambda)$ is called an
\emph{elementary Jordan block}. $\JJ_{1} (\lambda)$, which is still
diagonal, is called \emph{trivial}, while all $\JJ_n (\lambda)$ with
$n\geqslant 2$ are \emph{non-trivial}. Clearly, $N_n$ is
\emph{nilpotent}, with $N^{m}_{n} = \nix$ for all $m\geqslant n$, but
$N^{n-1}_{n}\ne \nix$, where $n$ is called the \emph{degree} of $N_n$.

In particular, simple matrices are cyclic, but sometimes too
restrictive (though still generic).  Here and below, we use the word
\emph{generic} for an attribute in a measure-theoretic sense, thus
referring to the property that the objects without this attribute form
a null set.  Further, the use of cyclic matrices is natural due to
Frobenius' theorem, which we recall as follows for the case of real
matrices; see \cite[Fact~2.10]{BS} and references given there for
details. Here and below, we use $[A,B \ts ] \defeq A\ts B - B\nts A$
for the \emph{commutator} of two matrices.

\begin{theorem}\label{thm:Frob}
  For\/ $B\in\Mat (d, \RR)$, the following properties are equivalent.
\begin{enumerate}\itemsep=2pt
\item $B$ is \emph{cyclic}, that is, the characteristic polynomial
  of\/ $B$ is, possibly up to a sign, also its minimal polynomial,
  so\/ $q_{_B} = (-1)^d \ts p_{\nts_B}$ with\/ $p_{\nts_B}$ and\/
  $q_{_B}$ from Eqs.~\eqref{eq:def-cp} and \eqref{eq:def-mp}.
\item There is a vector\/ $v\in\RR^d$ such that\/
  $\{ v, Bv, B^2v, \ldots, B^{d-1}v \}$ is a basis of\/ $\RR^d$.
\item $B$ is \emph{non-derogatory} \cite{HJ}, that is, one has\/
  $\mg (\lambda) = 1$ for all\/ $\lambda\in\sigma (B)$.
\item For each\/ $\lambda\in\sigma (B)$, the corresponding Jordan
  block of the JNF is\/ $\JJ_{\ma (\lambda)} (\lambda)$.
\item The centraliser of\/ $\nts B$,
  $\cent (B) \defeq \{ C \in \Mat (d, \RR) : [B,C \ts ]=\nix \}$, is
  Abelian.
\item One has\/ $\cent (B) = \RR [B]$, where\/ $\RR [B]$ is the
  polynomial ring generated by\/ $B$.  \qed
\end{enumerate}
\end{theorem}

The second property explains the name, and holds in this way because
eigenvalues and eigenvectors, as well as principal vectors, are real
or come in complex-conjugate pairs. The last property in
Theorem~\ref{thm:Frob} is the most powerful in our context, because
one has
\[
  \RR [B] \, = \, \big\langle \one, B, B^2, \ldots,
  B^{\deg (q_{\nts_B}) - 1}  \big\rangle_{\RR} \ts 
\]
by standard arguments. Property $(6)$ is equivalent to saying that
every matrix that commutes with the cyclic matrix $B$ is a polynomial
in $B$ with real coefficients and degree at most $d-\nts 1$; see
\cite[Thm.~3.2.4.2]{HJ} for details.  We shall make use of this
relation many times.

\begin{fact}\label{fact:one-is-special}
  For all\/ $M\in\cM_d$, one has\/ $1\in\sigma (M)$ together with\/
  $\mg (1) = \ma (1)$. In particular, there is no non-trivial Jordan
  block for\/ $\lambda = 1$.
  
  Further, the corresponding statement holds for generators, which is
  to say that any generator has\/ $0$ as an eigenvalue, with equal
  algebraic and geometric multiplicity.
\end{fact}

\begin{proof}
  The statement for $M$ is standard; see \cite[Thm.~13.10]{Gant} or
  \cite[Prop.~2.3{\ts}(2)]{BS}.
    
  If $Q$ is an arbitrary generator, one can always find a number
  $a > 0$ such that $\one + a \ts Q$ is Markov, from which the
  claim on generators is immediate.
\end{proof}

For a given matrix $B$, its exponential is defined by the power series
$\ee^B = \sum_{n=1}^{\infty} \nts\frac{1}{n\ts !}\ts B^{n}$, which
always converges; see \cite[Ch.~10]{Higham} for details. Due to the
Cayley--Hamilton theorem, it can be expressed as a polynomial in $B$
as well, and one always has $[\ee^B, B]=\nix$. Further, we have the
following property, which we will need repeatedly.

\begin{fact}[{\cite[Fact~2.15]{BS}}]\label{fact:diagonal}
  Let\/ $B\in\Mat (d,\CC)$. If\/ $\ee^B$ is diagonalisable, then so
  is\/ $B$.  \qed
\end{fact}

When $d=1$, the only Markov `matrix' is $M=1$, and the only
`generator' is $Q=0$. They connect via $1 = \ee^0$, which is the
unique relation here. For this reason, nothing of interest happens in
one dimension, and we shall generally assume $d\geqslant 2$ to avoid
trivialities.

For $d=2$, one has
\[
  \cM^{}_{2} \, = \, \biggl\{ {\scriptscriptstyle \biggl( 
      \begin{array}{@{}c@{\;}c@{}} 1{-}\ts a & a \\ b & 1{-}\ts b
    \end{array} \biggr) } : \ts a,b \in [\ts 0,1] \biggr\} \ts ,
\]
and the embedding problem is completely solved by Kendall's theorem,
which was not published by himself; see \cite[Prop.~2]{King} or
\cite[Thm.~1]{BS} for accounts with proofs.

\begin{theorem}[Kendall]\label{thm:Kendall}
  The Markov matrix\/
  $M = \left( \begin{smallmatrix} 1-\ts a & a \\ b & 1-\ts
      b \end{smallmatrix} \right)$ with\/ $a,b\in [ \ts 0,1]$ is
  embeddable if and only if\/ $\det (M) > 0$, which is equivalent to
  the condition\/ $0 \leqslant a+b <1$.  In this case, there is
  precisely one generator\/ $Q$ such that\/ $M=\ee^Q$, namely\/
  $Q = - \frac{\log(1-a-b)}{a+b} \bigl( M \nts\nts - \one \bigr)$.
  \qed
\end{theorem}

Here, we have the best possible situation, in that there is a simple
necessary and sufficient criterion, and, in the affirmative case, a
closed formula for the unique generator.  The reason for its
simplicity has to do with the fact that all $\one\ne M \in \cM_2$ are
simple (hence also cyclic) and have real spectrum. Then, embeddability
occurs if and only if $M$ has positive spectrum, which means
$\sigma (M) \subset \RR_{+}$; compare \cite{BS}. Parts of this
structure are still present when one looks into particular classes of
Markov matrices for $d>2$, but not in general. In fact, much of this
paper is concerned with the complications that emerge for $d=3$ and
$d=4$. General results for $d>4$ only exist for special types of
matrices; see the discussion in \cite{Higham,BS2} for details.

In particular, when $M\in\cM_d$ is cyclic and has positive spectrum,
the derived rate matrix $A \defeq M\nts - \one$ has spectral radius
$\varrho_{\nts_A} <1$. Then, the principal logarithm via the power
series
\begin{equation}\label{eq:log-def}
  \log (\one + A) \, = \sum_{m=1}^{\infty}
  \frac{(-1)^{m-1}}{m} \, A^m
\end{equation}
converges in norm, and defines a real matrix with zero row sums.  Note
that the actual calculation of $\log (\one+A)$ usually employs the
minimal polynomial of $A$, by which the logarithm can be expressed as
a polynomial in $A$; compare \cite{Higham,BS2}. Now, one has the
following general result \cite[Thm.~5.3 and Cor.~5.5]{BS2} on
(classic) embeddability; see also \cite[p.~38]{Higham}.

\begin{theorem}\label{thm:cyclic-1}
  Let\/ $M\in\cM_d$ be cyclic, with real spectrum. Then, $M$ is
  embeddable if and only if the following two conditions are
  satisfied, where\/ $A = M \nts - \one$.
\begin{enumerate}\itemsep=2pt  
\item The spectrum of\/ $M$ is positive, that is,
  $\sigma (M) \subset \RR_{+}$.
\item The matrix\/ $Q = \log(\one + A)$ is a generator.
\end{enumerate}
In this case, $Q$ is the principal matrix logarithm of\/ $M$, and the
embedding is unique, even in the sense that there is no other real
logarithm of\/ $M$ with zero row sums.  \qed
\end{theorem}

Here, the term \emph{real logarithm} of $M$ refers to any matrix
$R\in\Mat (d,\RR) $ with $\ee^R = M$. The following important
characterisation follows from \cite[Thms.~1 and 2]{Culver}, which is a
refinement of classic results from \cite[Ch.~5]{Gant}.

\begin{fact}\label{fact:Culver}
  A matrix\/ $B\in\Mat (d,\RR)$ has a real logarithm if and only if
  the following two conditions are satisfied.
\begin{enumerate}\itemsep=2pt
\item The matrix\/ $B$ is non-singular.
\item Each elementary Jordan block of the JNF of\/ $B$ that belongs 
  to a negative eigenvalue occurs with even multiplicity.
\end{enumerate}   
Further, when all eigenvalues of\/ $B$ are positive real numbers and
no elementary Jordan block occurs twice, the real logarithm of\/ $B$ 
is unique.  \qed
\end{fact}

Below, we refer to the existence part of Fact~\ref{fact:Culver} as
\emph{Culver's criterion}.  Note that, for the uniqueness statement,
the matrix $B$ need not be cyclic, as it can have two different Jordan
blocks for the same (positive) eigenvalue, and still satisfy the
condition.

\begin{remark}\label{rem:Culver}
  When Culver's criterion for negative eigenvalues is satisfied, or
  when any elementary Jordan block of a positive real eigenvalue of
  $B$ occurs more than once, there are \emph{uncountably} many real
  logarithms.  To see this, observe that one can start from one real
  logarithm with the matching JNF, and modify a pair of blocks by
  adding $2 \pi \ii k$ to the eigenvalue of one and its complex
  conjugate to the eigenvalue of the other, for some $0\ne k\in
  \ZZ$. This will not change the exponential of the matrix, but its
  symmetry. Indeed, while any matrix in $\GL(2,\RR)$ commutes with
  $\lambda \one^{}_{2}$, they need not (and generally will not)
  commute with $\diag (\lambda +2 \pi \ii k, \lambda - 2 \pi \ii k)$
  when $k\ne 0$, and similarly for block matrices with two equal
  versus two modified Jordan blocks.  We will meet one instance of
  this in Lemma~\ref{lem:extend-unique}.

  The other possible mechanism occurs if $B$ has a pair of complex
  conjugate eigenvalues, and is then a direct consequence of the
  structure of the complex logarithm. This results in a
  \emph{countable} set of real logarithms.  These are the two possible
  mechanisms for additional real logarithms, both of which will show
  up below; see \cite{Higham} or the discussion around the Corollary
  in \cite{Culver} for details.  It will be our task to identify the
  generators among them.  \exend
\end{remark}

Before we continue, let us formulate one simple necessary criterion
for embeddability due to Elfving \cite{Elfving}, and its consequence
on the spectrum of an embeddable matrix.

\begin{fact}\label{fact:Elf}
   If\/ $M\in\cM_d$ is embeddable, no diagonal element of\/ $M$
   can be zero. Further, $\lambda=1$ is the only eigenvalue of\/ $M$ 
   on the unit circle, and all other eigenvalues satisfy\/ 
   $\lvert \lambda \rvert < 1$.
\end{fact}

\begin{proof}
  Let $M=\ee^Q$ with a generator $Q$, and consider $M(t)= \ee^{t Q}$
  for $t\in[0,1]$, which is a continuous path in $\cM_d$ from
  $M(0)=\one$ to $M(1)=M$. By continuity, $M\bigl( \frac{1}{n}\bigr)$
  must have all diagonal elements strictly positive, for some
  sufficiently large $n\in\NN$.  Then, with
  $M(t) = \bigl( m^{}_{ij} (t) \bigr)_{1\leqslant i,j \leqslant d}$,
  we get
  $m^{}_{ii} (1) \geqslant \bigl( m^{}_{ii} (n^{-1})\bigr)^{n} > 0$
  for all $1\leqslant i \leqslant d$ as claimed.
  
  Clearly, all eigenvalues of a Markov matrix satisfy
  $\lvert \lambda\rvert \leqslant 1$ by the \emph{Perron--Frobenius}
  (PF) theorem; see \cite[Sec.~8.3]{HJ} for background. Now, if $M$ is
  embeddable, all diagonal elements are strictly positive. Let $p$ be
  the smallest of them, and consider $M - p \one$, which is still a
  non-negative matrix, so the modulus of all of its eigenvalues is
  bounded by its PF eigenvalue, which gives
  $\lvert \lambda - p \ts \rvert\leqslant 1-p$ for all
  $\lambda\in\sigma (M)$.  So, $\sigma (M)$ is contained in a closed
  disk of radius $1-p$ that lies inside the unit disk in such a way
  that it touches $1$, but no other point of the boundary, so
  $\lambda=1$ or $\lvert\lambda\rvert<1$.
\end{proof}

Now, (classic) embeddability of $M \in \cM_d$ means $M = \ee^Q$ for a
generator $Q$, and the set of all embeddable Markov matrices, for
fixed $d$, is denoted by $\emb$.  Clearly, embeddability is a special
case of the existence of a real logarithm, because each generator is a
real logarithm, while the converse is generally not true. This refers
to the hard part of the embedding problem, namely establishing the
required non-negativity conditions for the off-diagonal matrix
elements of the logarithm.  Let us begin with a simple result.

\begin{fact}\label{fact:one}
  The only generator\/ $Q = (q^{}_{ij} )^{}_{1\leqslant i,j \leqslant d}$ 
  that satisfies\/ $\ee^Q = \one^{}_{d}$ is\/ $Q=\nix^{}_{d}$.
\end{fact}

\begin{proof}
  The claim is trivial for $d=1$.  Since
  $\det (\one^{}_{d}) = 1 = \det (\ee^Q) = \ee^{\tr (Q)}$, we get
  $\tr (Q) = 0$ as $Q$ is real. Now, the generator property of $Q$
  implies that all $q^{}_{ii} \leqslant 0$, so
  $\tr (Q) = \sum_{i} q^{}_{ii}$ forces
  $0=q^{}_{ii} = -\sum_{i\ne j} q^{}_{ij}$ for all $i$, hence
  $q^{}_{ij} = 0$ for all $i\ne j$, and $Q=\nix^{}_{d}$ as claimed.
\end{proof}

Note that $\one^{}_{2}$ has uncountably many real logarithms, though
only one with zero row sums, the generator $\nix^{}_{2}$. As we shall
need this later for several cases with degenerate spectrum, let us
expand on this point a little. The equation $\ee^z = 1$ with $z\in\CC$
holds if and only if $z = 2 \pi \ii \ts k$ for some $k\in\ZZ$, and we
need the analogue of this in the sense that we want all solutions of
$\ee^R = \one^{}_{2}$ with $R\in\Mat (2,\RR)$. Such solutions exist,
as one sees from $\ee^{2 \pi I} = \one^{}_{2}$ with
$I = \left( \begin{smallmatrix} 0 & -1 \\ 1 &
    0 \end{smallmatrix}\right)$, and $\ee^{\gamma I} = \one^{}_{2}$
holds if and only if $\gamma = 2 \pi k$ for some $k\in\ZZ$, where
\[
    \exp (2 \ts \pi k I ) \, = \sum_{n=0}^{\infty}
    \myfrac{(2 \ts \pi k)^n}{n\ts !} \, I^n 
    \, = \, \cos (2 \ts \pi k) \one^{}_{2} +
    \sin (2 \ts \pi k) I \, = \, \one^{}_{2} \ts .
\]
The general result reads as follows. 

\begin{fact}\label{fact:real-log-2}
   The real logarithms of\/ $\one^{}_{2}$ are precisely the matrices\/
   $2 \pi k I_{x,y,z}$ with\/ 
\[
     I_{x,y,z} \, \defeq \,
    \begin{pmatrix} x & -z \\ y & -x \end{pmatrix},
\]   
where\/ $k\in\NN_0$ and\/ $x,y,z\in\RR$ such that\/ $yz-x^2=1$. The
parametrisation is unique, and the only real logarithm with zero row
sums is the one with\/ $k=0$.
\end{fact}

\begin{proof}
  The relation $\ee^R=\one^{}_{2}$ forces $R$ to be diagonalisable
  over $\CC$ by Fact~\ref{fact:diagonal}. As $R$ should be real, it
  must then be similar to $2 \pi k \, \diag (\ii, -\ii)$ within
  $\GL (2,\CC)$, for some $k\in\ZZ$. Writing
  $B = \left( \begin{smallmatrix} a & b \\ c & d
  \end{smallmatrix}\right)$ with complex entries and
  $\det(B) = ad-bc \ne 0$, a quick calculation shows that
  $B \, \diag (\ii, - \ii) B^{-1}$, which always lies in
  $\mathrm{SL} (2,\CC)$, must be of the form
  $I_{x,y,z} = \left( \begin{smallmatrix} x & -z \\ y &
    -x \end{smallmatrix}\right)$ with $x,y,z \in \CC$ and
  $yz-x^2=1$. We get all \emph{real} solutions by restricting the
  entries to $\RR$.

  By construction, these matrices are the most general real square
  roots of $-\one^{}_{2}$ in $\GL (2,\RR)$, which can also be
  calculated directly. Consequently, we get
  $\exp (2 \pi k I_{x,y,z}) = \one^{}_{2}$ for all $k\in\ZZ$ and
  $x,y,z\in\RR$ with $yz-x^2=1$. Since $I_{-x,-y,-z} = - I_{x,y,z}$
  and since $I_{x,y,z}=I_{x',y',z'}$ holds only for $x=x'$, $y=y'$ and
  $z=z'$, the claim is clear; compare \cite{CFR} for a similar
  treatment.
\end{proof}

Sometimes, it is advantageous to admit $k\in\ZZ$ and restrict $z$ to
be positive, which gives another unique parametrisation. Note that $y$
must then be positive as well. We shall make use of this freedom
without further notice. Every matrix $I_{x,y,z}$ from
Fact~\ref{fact:real-log-2} has eigenvalues $\pm \ii$ and is thus
simple.  Then, the commutation relation
$[ I_{x,y,z} , I_{x',y',z'} ] =\nix$, as a consequence of
Theorem~\ref{thm:Frob}{\ts}(6), implies $(x',y',z') = \pm
(x,y,z)$. Indeed, we have
\begin{equation}\label{eq:implication}
    I_{x',y',z'} \in \RR [ I_{x,y,z}] = \{ a \one^{}_{2} + b I_{x,y,z} :
    a,b \in \RR \}  \quad \Longrightarrow \quad
    a=0 \,\text{ and }\, b = \pm 1 \ts .
\end{equation}
This gives the following result. 

\begin{coro}\label{coro:real-log}
  Let\/ $\lambda \ne 0$ be a fixed real number. Then, the most general
  real logarithms of\/ $\lambda \one^{}_{2}$ are the following
  matrices.
\begin{enumerate}\itemsep=3pt
\item If\/
  $\lambda > 0 \! : \; \bigl\{ \log(\lambda) \one^{}_{2} + (2 k )\pi
  I_{x,y,z} \nts : k\in\ZZ \text{ and } x,y,z\in\RR \text{ with } z>0
  , \, yz-x^2=1 \bigr\}$.
\item If\/
  $\lambda < 0 \! : \; \bigl\{ \log \ts\lvert \lambda \rvert \ts
  \one^{}_{2} + (2 k+1) \pi I_{x,y,z} \nts : k\in\ZZ \text{ and }
  x,y,z\in\RR \text{ with } z>0 , \, yz-x^2=1 \bigr\}$.
\end{enumerate}  
In both cases, the parametrisation is unique.
\end{coro}

\begin{proof}
  Both claims follow from elementary calculations of the above type
  via \eqref{eq:implication}. Since similar ones also appear in
  \cite{CFR}, we leave the details to the interested reader.  The only
  additional subtlety occurs for the second case, where one has to
  show that a single matrix $I_{x,y,z}$ suffices, which is a
  consequence of the commutation properties in \eqref{eq:implication}.
\end{proof}

Already for $d=3$, the situation `complexifies', because $\one^{}_{3}$
also has uncountably many real logarithms with zero row sums
\cite[Rem.~2.11]{BS2}. This is deeply connected with some of the
difficulties of the embedding problem.  However, the following variant
of the uniqueness result in Fact~\ref{fact:Culver} will become
important below.

\begin{lemma}\label{lem:extend-unique}
  Let\/ $M\in\cM_d$ be a Markov matrix with the following properties.
\begin{enumerate}\itemsep=2pt
\item All eigenvalues of\/ $M$ are positive, so\/
  $\sigma (M) \subset \RR_{+}$.
\item For\/ $1\ne\lambda \in \sigma (M)$, no elementary Jordan block
  in the JNF occurs twice.
\item The multiplicity of\/ $\lambda = 1$ is\/ $\ts\ma (1) = 2$. 
\end{enumerate}  
\noindent
Then, there are uncountably many real logarithms, but only one of them
lies in\/ $\cA^{(d)}_{0}\!$, which then is the only candidate for a
generator.
\end{lemma}

\begin{proof}
  The part of the JNF connected with $\lambda=1$ is $\one^{}_{2}$, by
  Fact~\ref{fact:one-is-special}. Since $\sigma (M) \subset \RR_{+}$,
  we know from Fact~\ref{fact:Culver} that a real logarithm of $M$
  exists.  Let $M = \ee^R$ with real $R$, and consider the JNF of
  $R$. The only non-uniqueness emerges for the part that corresponds
  to $\one^{}_{2}$ in the JNF of $M$, and this can be $\nix^{}_{2}$ or
  any real matrix that is (complex) similar to
  $\diag (2 \pi \ii\ts k, - 2 \pi \ii\ts k)$ for some
  $0 \ne k \in \ZZ$; compare Remark~\ref{rem:Culver}.
  
  None of the logarithms that emerge from $k\ne 0$ can lie in
  $\cA^{(d)}_{0}\!$, since zero row sums for $R$ means that $0$ is an
  eigenvalue of $R$, then automatically with multiplicity $2$, and $R$
  is real. So, only $k=0$ leads to a real logarithm of $M$ from
  $\cA^{(d)}_{0}\!$, which is the principal matrix logarithm, for
  instance in the form of a convergent series. It then has the same
  type of JNF, which implies that it has the same centraliser as $M$
  itself, and we get the uniqueness as claimed.
\end{proof}

Let us close this section with a powerful consequence of
Theorem~\ref{thm:Frob}, which was discussed and exploited in detail in
\cite[Sec.~5, Thm.~5.3 and Cor.~5.5]{BS2}.

\begin{coro}\label{coro:alg}
  If\/ $M\in\cM_d$ is embeddable, so\/ $M=\ee^Q$ for some generator\/
  $Q$, the latter commutes with\/ $M$. When\/ $M$ is also cyclic, $Q$
  is an element of the non-unital real algebra generated by\/
  $A = M\nts - \one$, with\/ $\deg (q_{_A}) = d$. Thus,
  $Q \in \langle A, A^2, \ldots , A^{d-1}\rangle^{}_{\RR}$, which
  means
\[
     Q \, = \sum_{i=1}^{d-1} \alpha^{}_{i} \ts A^{i} 
\]  
for some\/
$\alpha^{}_{1}, \alpha^{}_{2}, \ldots , \alpha^{}_{\nts d-1} \in \RR$.
In particular, if the spectral radius of\/ $A$ satisfies\/
$\varrho_{\nts_A} < 1$, such a representation exists for the
convergent series\/ $\log (\one + A)$.  \qed
\end{coro}

Among the many results in the literature, the following existence and
uniqueness result from \cite{Cuth} sticks out because it is actually
useful in practice.

\begin{theorem}\label{thm:cuth}
  Let\/ $M\in \cM_d$ satisfy the inequality\/
  $\,\min^{}_{1\leqslant i \leqslant d} m^{}_{ii} > \frac{1}{2}$.
  Then, $Q = \log (\one + A)$ with\/ $A=M\nts - \one$ is well defined
  as a converging series.  Moreover, $M$ is embeddable if and only if
  this\/ $Q$ is a generator, and no other embedding is possible.
\end{theorem}

\begin{proof}
  Under the assumption, the spectral radius of $A$ satisfies
  $\varrho_{\nts_A} < 1$, as a simple consequence of Gershgorin's disk
  theorem \cite[Thm.~6.1.1]{HJ}, see also \cite[Ch.~14.5]{Gant}. Then,
  there is some proper (that is, sub-multiplicative) matrix norm
  $\|.\|$ such that $\|A\|<1$, and the series
\[
     \log (\one + A) \, = \sum_{n=1}^{\infty} \frac{(-1)^{n+1}}{n} A^n
\]   
is convergent in this norm, by a standard estimate in the form of a
Weierstrass M-test.  As all proper matrix norms on $\Mat (d,\CC)$ are
equivalent, convergence holds in any of them.

The limit clearly has zero row sums, but need not be a generator.
When it is, we get $M=\ee^Q$ and $M$ is embeddable, where the claimed
uniqueness (and absence of any other candidate for a generator)
follows from \cite[Cor.~1]{Cuth}.
\end{proof}

Another useful criterion for uniqueness is a consequence of the
Corollary on p.~ 530 of \cite{Cuth73}.

\begin{fact}\label{fact:cuth}
  Let\/ $M\in\cM_d$ be embeddable. If\/ $M$ also satisfies the
  condition
\[
     \det (M) \ts \min_{1\leqslant i \leqslant d} m^{}_{ii} 
     \, > \, \ee^{-\pi} \prod_{i=1}^{d} m^{}_{ii} \ts ,
\]   
the embedding is unique. In particular, this is the case if\/
$\det (M) > \ee^{-\pi} \approx 0.043{\ts}214$.  \qed
\end{fact}

This result is in line with several observations that a
\emph{logarithmic} scale would be more natural, which is sometimes
referred to as `log det' giving the proper intrinsic time scale of the
problem \cite{Good}. This can also be seen from the solution of the
differential equation for the determinant of $M (t)$; compare
Eq.~\eqref{eq:trace} below. In general, an embedding need not be
unique. In such a case, one particular consequence of
Theorem~\mbox{\ref{thm:Frob}{\ts}(5)} is the following; see
\cite[Cor.~10]{Davies} and \cite[Fact~2.14]{BS} for more.

\begin{coro}\label{coro:gen-commute}
  If a cyclic matrix\/ $M\in\cM_d$ admits more than one embedding, the
  corresponding generators must commute with one another, and with\/
  $M$.  \qed
\end{coro}

We are now ready to embark on the embedding problem for $d=3$ and
$d=4$, where further notions will be introduced where and when we need
them.

\section{Embedding in three dimensions}\label{sec:three}

Given a general $M\in\cM^{}_{3}$, where we know that $1$ is an
eigenvalue, it is most systematic to classify the cases according to
$\deg (q_{_M})$, the degree of the minimal polynomial of $M$.  This is
an element of $\{ 1, 2, 3\}$, where $\deg (q_{_M}) = 1$ is trivial
because this implies $M=\one$, which brings us back to
Fact~\ref{fact:one}.  From now on, it is often advantageous to work
with $A=M\nts - \one$, where $q_{_A} (x) = q_{_M} (x+1)$ implies
$\deg (q_{_A}) = \deg (q_{_M})$. Recall that $A$ (and hence $M$) is
diagonalisable if and only if $q_{_A}$ has no repeated factor; see 
\cite[Cor.~3.3.8]{HJ}.

\subsection{Cases of degree 2.}\label{sec:d-3.2}

When $\deg (q_{_M}) = 2$, we must have
$q_{_M} (x) = (x-1) (x-\lambda)$ with $\lambda \in [-1, 1)$, where
embeddability excludes $\lambda = 0$ because $\det (M)$ must be
positive; compare \cite[Prop.~2.1]{BS}. Since $q_{_M}$ has no repeated
factor, the matrix $M$ is always diagonalisable in this case.  Now, we
have to distinguish two subcases, namely $\ma (1) = 2$ and
$\ma (\lambda) = 2$.

If $\ma (1) = 2$, the JNF of $M$ is $\diag (1, 1, \lambda)$, with
$\lambda\ne 1$. Since $A=M\nts - \one$ is then similar to
$\ts\diag (0,0,\lambda \ts{-} 1)$, a quick calculation of
$T \ts \diag (0,0,\lambda \ts{-} 1) \ts T^{-1}$ with
$T\in \GL (3, \RR)$ shows that $A$ must be of the form
\[
    A \, = \, (\lambda - 1) \,
    ( \alpha^{}_{1}, \alpha^{}_{2}, \alpha^{}_{3})^{\mathsf{T}}
     \!\cdot (a^{}_{1}, a^{}_{2}, a^{}_{3})
\]
subject to the condition $\sum_{i=1}^{3}\alpha^{}_{i} a^{}_{i} = 1$
(to get the right eigenvalues), $\sum_{i=1}^{3} a^{}_{i} = 0$ (to have
zero row sums), and the obvious sign conditions to make $A$ a rate
matrix.  One thus finds that the most general Markov matrix here has
the form
\[
   M \, = \, \one +  (\lambda-1) 
     \bigl( \alpha^{}_{i} a^{}_{j} \bigr)_{1\leqslant i,j \leqslant 3}
\]
with the above conditions on the parameters, and subject to the
remaining conditions that $M$ is actually Markov. In particular,
$\lvert \alpha^{}_{i} a^{}_{j} \rvert \leqslant 1$ for all $i,j$
together with $\alpha^{}_{i} a^{}_{j}$ non-negative for $i=j$ and
non-positive otherwise. Further calculation shows that this class
comprises the matrices of the form $1 \oplus M'$ with $M'\in\cM_2$ and
its relatives that are obtained under a permutation of the three
states, and all Markov matrices with a single non-trivial row.

Embeddability of $M$ forces $\lambda \in (0,1)$ by the determinant
condition, since $0 < \det (M) = \lambda$. Here, we are in the
situation of Lemma~\ref{lem:extend-unique}, so we know that only one
real logarithm of $M$ exists, with JNF
$\diag \bigl(0,0,\log(\lambda)\bigr)$.  Since $A=M\nts -\one$ has
spectral radius $\varrho_{\nts_A} < 1$, we get
\begin{equation}\label{eq:simplification}
  R \, = \, \log (\one + A) \, = \,
  \frac{- \log (\lambda)}{1-\lambda} A \ts ,
\end{equation}   
where the second identity is a consequence of $A^2 = (\lambda - 1) A$.
Since $\lambda \in (0,1)$, the matrix $R$ is always a generator.
\smallskip

Next, consider $\ma (\lambda) = 2$ with $- 1 \leqslant \lambda < 1$
and $\lambda\ne 0$. Since $A=M\nts - \one$ is similar to
$\ts\diag (0, \lambda \ts{-}1, \lambda \ts{-}1)$, a general similarity 
transform shows that $A\in \cA^{(3)}_{0}$ must be of the form
\[
     A \, = \, (\lambda-1) \bigl( \one - ( 1,1,1 )^{T}
     \!\cdot (a^{}_{1}, a^{}_{2}, a^{}_{3}) \bigr) ,
\]
again with $\sum_{i=1}^{3} a^{}_{i} = 0$. Via a trivial
re-parametrisation, one sees that the most general Markov matrices in
this case are the \emph{equal-input} matrices
\begin{equation}\label{eq:eq-input-3}
     M \, = \, M_{c} \, = \: 
    (1- c) \ts \one + C ( c^{}_{1} , c^{}_{2} ,  c^{}_{3} ) \ts ,
\end{equation}
where $ C ( c^{}_{1} , c^{}_{2} , c^{}_{3} )$ is the matrix with three
equal rows $( c^{}_{1} , c^{}_{2} , c^{}_{3} )$ and summatory
parameter $c = c^{}_{1} + c^{}_{2} + c^{}_{3} = 1-\lambda$, where
$c\ne 0$ to exclude $M=\one$.  Clearly, we need all
$c^{}_{i} \geqslant 0$ together with $c \leqslant 1 + c^{}_{i}$ for
all $i$; see \cite{Steel,BS,BS2} for more on this class.  Here, we
know that uncountably many real logarithms exist, by
Fact~\ref{fact:Culver} and the discussion following it, and each of
them must be diagonalisable by Fact~\ref{fact:diagonal}.  Let $R$ be
any of them.

When $\lambda > 0$, the SMT from Eq.~\eqref{eq:SMT} forces $R$ to have
eigenvalues $0$ and $\mu^{}_{\pm} = \log (\lambda) \pm 2 \pi \ii\ts k$
for some $k\in\ZZ$. When $k=0$, we have $\cent (R) = \cent (M)$, by
the argument used in the proof of Lemma~\ref{lem:extend-unique}, which
means that only one real logarithm corresponds to this choice, namely
$R = \log (\one + A)$, which is a generator. Since
$A^2 = (\lambda - 1) A$ holds also in this case, we get the formula
from Eq.~\eqref{eq:simplification} again.

For each $k\ne 0$, the matrix $R$ is simple, and $\cent (R)$ is a true
subalgebra of $\cent (M)$. Consequently, the set of real logarithms
corresponding to any such choice of $k$ is uncountable: For every real
logarithm $R$, also $S\nts\nts R \ts S^{-1}$ is one, for arbitrary
invertible $S \in \cent (M)$. There can be further generators among
them. So far, we have the following.

\begin{lemma}\label{lem:3-deg-2}
  Let\/ $M\in\cM^{}_{3}$ have\/ $\deg (q_{_M} ) = 2$ and spectrum\/
  $\sigma (M) = \{ 1, \lambda\}$ with\/ $\lambda \ne 1$. Then,
  $\lambda \in \RR$, and we get the following case distinction.
  \begin{enumerate}\itemsep=2pt
  \item If\/ $M$ has JNF\/ $\diag (1 , 1 , \lambda)$, it is embeddable
    if and only if\/ $\lambda \in (0,1)$.  In this case, one has\/ the
    generator\/ $Q$ from Eq.~\eqref{eq:simplification}, and the
    embedding\/ $M=\ee^Q$ is unique.
  \item If\/ $M$ has JNF\/ $\diag (1 , \lambda , \lambda)$ with
    $\lambda>0$, hence\/ $\lambda \in (0,1)$, it is embeddable with
    the generator\/ $Q$ from Eq.~\eqref{eq:simplification}, but
    further solutions may exist.  \qed
  \end{enumerate} 
\end{lemma}

Note that, in view of Eq.~\eqref{eq:eq-input-3}, the second claim of
Lemma~\ref{lem:3-deg-2} is equivalent with the previous result from
\cite[Prop.~2.12]{BS2} that an equal-input Markov matrix with
summatory parameter $0<c<1$ is always embeddable, even with an
equal-input generator, where $\lambda = 1-c$. However, it need not
be unique, as the following example shows.

\begin{example}\label{ex:extremal-case}
  Let $M$ be the equal-input matrix from \cite[Ex.~4.3]{BS}, which has
  JNF $\diag (1, \lambda, \lambda)$ with
  $\lambda = - \ee^{-\pi \sqrt{3}}<0$ and summatory parameter
  $c = 1+\ee^{-\pi\sqrt{3}} > 1$. Explicitly, it is
  $M = \one + (1 + 3 \delta) J^{}_{3}$ with
  $\delta = \frac{2}{3} \ee^{-\pi \sqrt{3}}$ and the matrix $J^{}_{3}$
  from Lemma~\ref{lem:star} below. This $M$ is embeddable via two
  different circulant generators, as we shall derive later in
  Example~\ref{ex:max-c}.
  
  Now, the matrix $M^2$ is still doubly embeddable in a circulant way,
  and is also still equal input, but now with summatory parameter
  $c' = c (2-c) = 1 - \ee^{-2 \pi \sqrt{3}} < 1$. Consequently, by
  \cite[Thm.~4.6]{BS}, it is additionally embeddable via an
  equal-input generator,\footnote{Later on, we refer to this property
    as equal-input embeddable, and similarly for other model classes.}
  and we thus have an example with $c<1$ and three distinct
  embeddings.  \exend
\end{example}

\begin{remark}
  In the second case of Lemma~\ref{lem:3-deg-2}, all other embeddings
  can be found with the help of Corollary~\ref{coro:real-log}, though
  none of them can be of equal-input type.  Indeed, fixing a matrix
  $T\in\GL(3,\RR)$ such that
  $M = T^{-1} \diag (1,\lambda,\lambda) \ts T$, one has to identify
  all generators among
  $T^{-1} \bigl( 1 \oplus ( \log (\lambda) \one^{}_{2} + 2 \pi k
  I_{x,y,z} ) \bigr) T$ with $k\in\ZZ$ and $x,y,z\in\RR$ subject to
  $z>0$ and $yz-x^2=1$. In fact, by \cite[Lemma~3.1]{CFR}, one only
  has to consider integers $k$ with
\[
   \lvert k \rvert \, \leqslant \, \frac{ \ts \lvert \log(\lambda)
     \rvert \ts} {2 \pi \sqrt{3}} \ts .
\]      
In particular, if
$\lambda > \ee^{- 2 \pi \sqrt{3}} \approx 1.877{\ts}853 \cdot \nts
10^{-5}$, there can be no further candidate except $k=0$, and the
embedding via $\log (\one+A)$ is unique.  This condition can also be
expressed via the determinant, as
$\det (M) = \lambda^2 > \ee^{- 4 \pi \sqrt{3}} \approx 3.526{\ts}333
\cdot\nts 10^{-10}$, which is a further improvement over
\cite[Cor.~3.3 and Table~1]{CFR} for this more special case.  \exend
\end{remark}

Finally, if $M$ has a negative eigenvalue $\lambda$, this must have
even multiplicity for embeddability, and $M$ must be diagonalisable
(due to Fact~\ref{fact:Culver}). So, it has $\deg (q_{_M})=2$, which
is why this case occurs here. It was the last case for $d=3$ to be
solved \cite{Joh, Carette}. However, the answer is rather tricky and
not practically useful, wherefore we do not recall those details.
Instead, since the only Markov matrices with this type of JNF are the
equal-input matrices from Eq.~\eqref{eq:eq-input-3} with $c>1$, we can
use their structure to get a simpler and constructive result. In fact,
we also know $c\leqslant \frac{3}{2}$ in this case, since
$c\leqslant 1 + c^{}_i$ for all $i$, hence
$-\frac{1}{2}\leqslant \lambda < 0$. The key now is to control the
matrices that commute with $M$.

If an equal-input matrix $M_c$ has a real logarithm, so $M_c = \ee^R$
for some $R\in\Mat (3,\RR)$, one has $[M_c, R \ts ]=\nix$. If
$M_c = (1-c) \ts \one + C$ with $C=C(c^{}_{1}, c^{}_{2}, c^{}_{3})$
and $c=c^{}_{1} + c^{}_{2} + c^{}_{3}$ as above, one clearly has
$[M_c, R \ts ]=\nix$ if and only if $[C,R \ts ]=\nix$.  Since we are
only interested in matrices $R$ with zero row sums, we define the
\emph{commutant} of $M_c$ as
\[
  \comm (M_c) \, \defeq \, \{ A \in \cA^{(3)}_{\ts 0} :
  [ M_c, A] = \nix \} \ts ,
\]
where one has $\comm (M_c) = \comm (C)$ with the $C$ from
above. Clearly, one has the trivial case that
$\comm (\one^{}_{3}) = \cA^{(3)}_{\ts 0}\!$, which is a
six-dimensional algebra, while it is four-dimensional for all other
equal-input matrices, as can be derived from their diagonal form. In
view of the situation at hand, we now look more closely at the case
that all $c^{}_{i} > 0$.

\begin{lemma}\label{lem:comm}
  Let\/ $M_c$ be the matrix from Eq.~\eqref{eq:eq-input-3}, with all\/
  $c^{}_i > 0$. Then, the algebra\/ $\comm (M_c)$ is generated by the
  four matrices
\[
    \begin{pmatrix} 0 & 0 & 0 \\ 0 & * & c^{}_{3} \\
     0 & c^{}_{2} & * \end{pmatrix}  , \;
     \begin{pmatrix} * & 0 & c^{}_{3} \\ 0 & 0 & 0 \\
    c^{}_{1} & 0 & * \end{pmatrix} , \;
  \begin{pmatrix} * & c^{}_{2} & 0 \\ c^{}_{1} & * & 0 \\
    0 & 0 & 0 \end{pmatrix}  \quad\text{and}\quad
   \begin{pmatrix} * & \alpha & -\gamma \\
     -\alpha & * & \beta \\ \gamma & -\beta & * \end{pmatrix} ,
\]
with\/ $\alpha = (c^{}_{1} + c^{}_{3})(c^{}_{2} + c^{}_{3})$,
$\beta =(c^{}_{1} + c^{}_{2})(c^{}_{1} + c^{}_{3}) $ and
$\gamma =(c^{}_{1} + c^{}_{2})(c^{}_{2} + c^{}_{3}) $.  Here, $*$
always denotes the unique real number to assure row sum\/ $0$. These
matrices are linearly independent over\/ $\RR$, and\/ $\comm (M_c)$ is
a subalgebra of\/ $\cA^{(3)}_{\ts 0}$ of dimension\/ $4$.
\end{lemma}

\begin{proof}
  Let $C=C(c^{}_{1}, c^{}_{2}, c^{}_{3})$ and observe that
  $R \ts C =\nix$ holds for every $R$ with zero row sums. Then,
  $[C, R]=\nix$ means $C R=\nix$, which holds if and only if
  $(c^{}_{1}, c^{}_{2}, c^{}_{3})$ is a left eigenvector of $R$ with
  eigenvalue $0$.

  This eigenvector property is satisfied for each of the four matrices
  given, as one can check by a simple calculation.  As $c^{}_{i} > 0$
  for all $i$, the four matrices are all non-zero and indeed linearly
  independent over $\RR$, wherefore they span a four-dimensional
  subalgebra of $\cA^{(3)}_{\ts 0}\!$, which must be contained in
  $\comm (M_c)$.

  Since $\cA^{(3)}_{\ts 0}$ has dimension $6$, and since the left
  eigenvector condition results in two independent constraints due to
  the zero row sum property of $R$, we see that $\comm (M_c)$ has
  dimension $4$, and is thus spanned by the matrices given.
\end{proof}

The commutant can also be calculated explicitly for any non-zero
vector $(c^{}_{1}, c^{}_{2}, c^{}_{3})$ via the left eigenvector
condition, and always gives a four-dimensional commutant, then with a
slightly more complicated parametrisation; we skip further details at
this point.

\begin{lemma}\label{lem:comm-2}
  Let\/ $M_c$ be the matrix from Eq.~\eqref{eq:eq-input-3}, with all\/
  $c^{}_i \geqslant 0$ and\/ $c>0$, and assume that\/
  $Q \in \comm (M_c)$ is a generator such that\/ $\ee^Q$ is an
  equal-input matrix. Then, we have\/
  $\ee^Q = (1-r)\ts \one + \frac{r}{c} C$ for some\/ $r\geqslant 0$,
  with the matrix\/ $C$ and summatory parameter\/ $c$ from\/ $M_c$.
\end{lemma}

\begin{proof}
  If $\ee^Q$ is of equal-input type, we have
  $\ee^Q = (1-r) \ts \one + \widetilde{C}$ with
  $\widetilde{C} = C (r^{}_{1}, r^{}_{2}, r^{}_{3})$ and all
  $r_i \geqslant 0$, where $r$ is the summatory parameter. Since $Q$
  commutes with $M_c$ by assumption, we obtain $[ \ee^Q, M_c ]=\nix$,
  as we explained prior to Fact~\ref{fact:diagonal}, and thus
  $\bigl[ C, \widetilde{C}\, \bigr]=\nix$. As
  $C \ts \widetilde{C} = c \ts \widetilde{C}$ and
  $\widetilde{C} C = r \ts C$, we get
  $\widetilde{C} = \frac{r}{c} \ts C$ and the claim follows.
\end{proof}

Our strategy is now the following. Since we only need further insight
into the embedding properties of $M_c$ with
$1 < c \leqslant \frac{3}{2}$, we assume $c>1$, which forces $c_i > 0$
for all $i$. So, consider the matrix
$C = C(c^{}_{1}, c^{}_{2}, c^{}_{3})$, and let $R \in \comm (C)$ be a
matrix such that $\ee^R$ is equal input, hence
$\ee^R = (1-r) \one + \widetilde{C}$, where we know from
Lemma~\ref{lem:comm-2} that $\widetilde{C} = \frac{r}{c} \ts C$ for
some $0 \leqslant r \ne 1$, which is the new summatory
parameter. Since the cases with $r\in [0,1)$ are trivially embeddable,
we now concentrate on the case $r>1$ and determine the maximal value
for embeddability, $r^{\max}$. This way, we will find the full
embeddability range $[0,1) \cup (1, r^{\max} \ts ] $ for the direction
defined by $C$, respectively by $(c^{}_{1}, c^{}_{2}, c^{}_{3})$.

Given $(c^{}_{1}, c^{}_{2}, c^{}_{3})$ with $c>1$ as above, let
$Q^{}_1, Q^{}_2, Q^{}_3$ and $R^{}_{0}$ denote the four matrices from
Lemma~\ref{lem:comm-2}, in the order as given there. All $Q^{}_i$ are
generators, while $R^{}_{0} \in\cA^{(3)}_{0}$ never is. It is
immediate that $Q=x\ts Q^{}_1 + y\ts Q^{}_2 + z\ts Q^{}_3 + w R^{}_0$
is a generator if and only if
\begin{equation}\label{eq:gen-cond}
    x \, \geqslant \, \max \Big\{ \myfrac{\beta w}{c^{}_{2}}, 
       -\myfrac{\beta w}{c^{}_{3}} \Big\} , \quad
    y \, \geqslant \,  \max \Big\{ \myfrac{\gamma w}{c^{}_{3}}, 
       -\myfrac{\gamma w}{c^{}_{1}} \Big\} , \quad \text{and} \quad
    z \, \geqslant \, \max \Big\{ \myfrac{\alpha w}{c^{}_{1}}, 
       -\myfrac{\alpha w}{c^{}_{2}} \Big\} .
\end{equation}
In particular, $x,y,z\geqslant 0$, and they are even strictly positive
unless $w=0$.

\begin{lemma}\label{lem:comm-3}
  Consider\/
  $Q = x\ts Q^{}_1 + y\ts Q^{}_2 + z\ts Q^{}_3 + w R^{}_{0}$ with the
  matrices from Lemma~\textnormal{\ref{lem:comm}} as defined
  above. Assume that\/ $Q$ is a generator so that\/ $M=\ee^Q$ has a
  double negative eigenvalue, say\/ $-\lambda < 0$. Then, $M$ is an
  equal-input matrix with parameters\/
  $\frac{r}{c} (c^{}_{1}, c^{}_{2}, c^{}_{3})$ for some\/ $r>1$.
\end{lemma}

\begin{proof}
  Since $\sigma (M) = \{ 1, -\lambda, -\lambda \}$ by assumption, this
  time viewed as a multi-set, and since $Q$ is a real matrix, we must
  have $\sigma (Q) = \{ 0, \eta \pm \ts (2 k\ts {+}1) \pi \ii \}$ for
  some $k \in \NN_0$ by the SMT from Eq.~\eqref{eq:SMT}, where
  $\eta<0$.  So, $Q$ is simple and hence diagonalisable, with
  $\lambda = \ee^{\eta}$. Clearly, $M$ is then diagonalisable as well,
  and its JNF is $\ts\diag (1 , -\lambda , -\lambda )$.
  
  Now, we can repeat the little calculation that led to
  Eq.~\eqref{eq:eq-input-3} to conclude that $M$ is indeed equal
  input.  Since we know that $Q$ commutes with the original $C$-matrix
  defined by the specified $c_i$, an application of
  Lemma~\ref{lem:comm-2} establishes the claim.
\end{proof}

This means that we can proceed via the SMT. If $Q$ is the generator
from Lemma~\ref{lem:comm-3}, its spectrum is
$\sigma (Q) = \{ 0, -\Delta \pm s \}$ with
\begin{equation}\label{eq:del+sq}
\begin{split}
   \Delta \, & = \, \myfrac{1}{2} \bigl( (c^{}_{2} + c^{}_{3} ) x +
   (c^{}_{1} + c^{}_{3}) y + (c^{}_{1} + c^{}_{2} ) z \bigr)
   \, \geqslant \, 0 \quad \text{and} \\
   s^2 \, & = \, \myfrac{1}{4} \bigl( (c^{}_{1} + c^{}_{3}) (x-y)
   - (c^{}_{1} + c^{}_{2}) (x-z) \bigr)^2
   + c^{}_{2} c^{}_{3} (x-y)(x-z) \\ & \qquad \;\;\,
   + (c^{}_{1} + c^{}_{2} + c^{}_{3}) \ts
   \bigl( c_1^2  (z-y) + c_2^2 \ts (x-z)  + c_3^2 (y-x) \bigr) w
     \\[1mm] & \qquad\;\;\,
     - 2 \ts (c^{}_{1} + c^{}_{2} + c^{}_{3}) \ts
     (c^{}_{1} + c^{}_{2})(c^{}_{1} + c^{}_{3}) 
     (c^{}_{2} + c^{}_{3}) \ts w^2 \\
     & \eqdef \, \chi + \psi \ts w + \varphi \ts w^2 ,
\end{split}   
\end{equation}
as can easily be checked with any computer algebra system.  Note that
$s^2$ only depends on differences of the parameters $x,y,z$, which
will become important shortly. Now, we fix $s^2 = - (2 k+1)^2 \pi^2$
for some $k\in\NN_0$, which gives $\ee^{\pm s} = -1$, and
$\sigma (\ee^Q)$ thus is the multi-set $\{ 1, -\lambda, -\lambda\}$
with $\lambda = \ee^{-\Delta}$. We need the largest possible value of
$\lambda$ and thus the smallest value for $\Delta$, where all
inequalities from \eqref{eq:gen-cond} have to be satisfied for $Q$ to
be a generator. When $w=0$, we see that $s^2=\chi$ is the sum of a
square and one extra term, which can be done in three different ways
(only one of which is shown above). The other two choices give
$(y-x)(y-z)$ and $(z-x)(z-y)$ in the extra term, respectively, which
ensures positivity according to $\max (x,y,z)$. This implies
$s^2 \geqslant 0$ whenever $w=0$. So, we need $w\ne 0$ and then
$x,y,z > 0$ to ensure that $Q$ is a generator.

As $c^{}_{i} >0$ for all $i$, no partial derivative of $\Delta$
can vanish (except the trivial one with respect to $w$), while those
of $s^2$ are identically $0$, because $s^2 = - (2k+1)^2 \pi^2$ is
constant. The minimum we are looking for, as a function of $x,y,z$, is
then not a local one, but must lie on the boundary of the region that
guarantees the generator property of $Q$. We thus need the correct
value for $w$. Now, $s^2 = - (2k+1)^2 \pi^2$ leads to a quadratic
equation for $w$, with the two solutions
\[
   w \, = \, \myfrac{-1}{2 \varphi} \Bigl(\psi \pm 
   \sqrt{\psi^2 - 4 \ts \varphi \bigl( \chi +
     (2k+1)^2 \pi^2 \bigr)}\, \Bigr) ,
\]
where $\varphi, \psi, \chi$ are defined in \eqref{eq:del+sq}.

\begin{example}\label{ex:max-c}
  Consider the special case that
  $c^{}_{1} = c^{}_{2} = c^{}_{3} = \frac{c}{3} > 0$, where
  $\alpha = \beta = \gamma = \frac{4}{9} \ts c^2$ and then
  $\varphi = - \frac{16}{27} c^4$, $\psi = 0$ and
  $\chi = \frac{c^2}{9} \bigl( (x-y)^2 + (x-z)^2 - (x-y)(x-z)\bigr)$.
  Now, the condition $w_x = w_y = w_z = 0$ for a stationary point is
  equivalent with $x=y=z$, which gives $\chi = 0$ and thus the two
  solutions
\[
       w \, = \, \pm  \ts \myfrac{9 \ts (2k{+}1) \ts \pi}
            {4 \ts c^2 \sqrt{3}\ts } \ts ,
\]
where we now choose $k=0$ to minimise the modulus of $w$.  Since
$x\geqslant \frac{ \pi \sqrt{3} }{ c }$ from \eqref{eq:gen-cond}, the
minimal value of $\Delta$ is $\Delta^{\min} = c \ts x = \pi \sqrt{3}$,
which means $\lambda = - \ee^{-\Delta^{\min}} = 1 - r^{\max}$ and
hence $r^{\max} = 1 + \ee^{-\pi \sqrt{3}}$, which we know from
Example~\ref{ex:extremal-case}.
  
There are two solutions for the generator $Q$ now, namely the
circulant matrices
\[
    \myfrac{2\ts \pi}{\sqrt{3}}\begin{pmatrix}
    -1 & 1 & 0 \\ 0 & -1 & 1 \\ 1 & 0 & -1 \end{pmatrix}
    \quad \text{and} \quad 
    \myfrac{2\ts \pi}{\sqrt{3}}\begin{pmatrix}
    -1 & 0 & 1 \\ 1 & -1 & 0 \\ 0 & 1 & -1 \end{pmatrix} .
\]  
One indeed gets $\ee^Q = \one + (1+3\ts \delta) J^{}_{3}$ with
$\delta = \frac{2}{3} \ee^{-\pi\sqrt{3}}$ and $J^{}_{3}$ as in
Lemma~\ref{lem:star} below. This extremal case is known from
\cite[Ex.~4.3]{BS}, see also Example~\ref{ex:extremal-case}, but here
we directly obtain \emph{two} embeddings.  \exend
\end{example}

Let us now sketch the general argument. Select a parameter triple
$(c^{}_1, c^{}_2, c^{}_3 )$ in the positive cone, subject to the
condition $c>1$. We now determine the maximal point on the ray defined
by this triple that is still embeddable. To this end, we fix
$s^2 = - (2k+1)^2\pi^2$ with $k\in\NN_0$, so $M_c$ has two negative
eigenvalues. This gives an implicit equation for $w$ as a function of
$x, y, z$.  The stationarity condition $w_x = w_y = w_z = 0$, which we
will justify a little later, now gives that both $x-y$ and $x-z$ are
proportional to $w$, the details of which are best calculated with a
computer algebra system. Inserting this into $s^2$ leads to
\[
   s^2 \, = \, - \frac{c \ts (c^{}_1 \nts + c^{}_2)^2 
      (c^{}_1 \nts + c^{}_3)^2 
      (c^{}_2 + c^{}_3)^2}{4 \ts c^{}_1 c^{}_2 \ts c^{}_3} \ts w^2
      \, = \, - (2k+1)^2 \pi^2 ,
\]
and hence to the always real solutions (with $k\in\NN_0$)
\begin{equation}\label{eq:w-k}
     w \, = \, \pm \frac{2 \ts (2 k{+}1) \ts \pi \, 
        \sqrt{ \frac{c^{}_{1_{\vphantom{\chi}}} \nts
        c^{}_2 \ts c^{}_3}{c} } }{(c^{}_1 \nts + c^{}_2) 
        (c^{}_1 \nts + c^{}_3) (c^{}_2 + c^{}_3)} \ts .
\end{equation}
This can now again be inserted into the expressions for $x-y$ and
$x-z$ to eliminate $w$, in two different ways, for the two choices of
the sign of $w$.

Next, to minimise $\Delta$, we want to take the \emph{minimal} value
of $x$, $y$, and $z$ that is allowable by the generator condition
\eqref{eq:gen-cond}. This forces $k=0$, but which of the three
conditions in \eqref{eq:gen-cond} is pivotal here depends on the
parameters $c_i$. For $x$, the condition reads
\[
  x \, \geqslant \, (c^{}_1 \nts + c^{}_2) (c^{}_1 \nts + c^{}_3) \max
  \Bigl( \myfrac{w}{c^{}_{2}} , \myfrac{-w}{c^{}_{3}} \Bigr).
\]
Taking the minimal value and determining $y$ and $z$ from the
stationarity condition then gives
$\Delta \geqslant \pi \sqrt{c\ts} \, c^{}_{1} / \!
\sqrt{c^{\vphantom{a}}_{1} \nts c^{}_2 \ts c^{}_3}$ upon inserting,
and similarly for the other two paths that start from $y$ or from
$z$. Since all conditions have to be satisfied to obtain a generator,
we find
\begin{equation}\label{eq:Delta}
    \Delta^{\min} \, = \, \frac{\pi \ts \kappa \sqrt{c\ts}}
    {\sqrt{c^{\vphantom{a}}_{1} c^{}_2 \ts c^{}_3}} \, , 
    \qquad \text{with } \, \kappa = \max (c^{}_1, c^{}_2, c^{}_3) \ts ,
\end{equation}
which holds for \emph{both} choices $w = \pm \lvert w \rvert$.  The
corresponding two generators read
\begin{equation}\label{eq:gen-ei}
  Q^{}_{\pm} \, = \, \frac{\pi}{\sqrt{(c^{}_1 \nts + c^{}_2 + c^{}_3)
  \ts c^{}_1 c^{}_2 \ts c^{}_3 \ts }\,} \begin{pmatrix}
  -\kappa \ts (c^{}_2 + c^{}_3) & c^{}_2 (\kappa \pm c^{}_3) 
       & c^{}_3 (\kappa \mp c^{}_2) \\
  c^{}_1 (\kappa \mp c^{}_3) & -\kappa \ts (c^{}_1 \nts + c^{}_3) 
       & c^{}_3 (\kappa \pm c^{}_1) \\
  c^{}_1 (\kappa \pm c^{}_2) & c^{}_2 (\kappa \mp c^{}_1) 
      & -\kappa \ts (c^{}_1 \nts + c^{}_2)  \end{pmatrix},
\end{equation}
which is again best checked by a computer algebra system.  For
$c^{}_1 = c^{}_2 = c^{}_3 = \frac{c}{3}$, this reduces to the matrices
from Example~\ref{ex:max-c} for the constant-input case.  Note that
the expressions for $\Delta^{\min}$ and for $Q^{}_{\pm}$ are
positively homogeneous of degree $0$ in the $c^{}_{i}$, so only depend
on the \emph{direction} defined by the vector
$(c^{}_1 , c^{}_2, c^{}_3 )$.

The final step now consists in observing that the relation
$M_c = \exp (Q^{}_{\pm})$, via the negative eigenvalue
$\lambda = - \ee^{-\Delta^{\min}}$ and $c = 1 - \lambda$, gives the
consistency condition $c = 1+ \ee^{-\Delta^{\min}}\!$. If this is not
satisfied, one has to replace the $c^{}_i$ by the unique values
\[
    c^{\max}_{i} \, = \, \frac{1 + \ee^{-\Delta^{\min}}}
    {c^{}_1 \nts + c^{}_2 + c^{}_3} \, c^{}_{i} \ts ,
\]
which give the extremal point of embeddability on the ray defined by
$(c^{}_1 , c^{}_2 , c^{}_3 )$. We thus have the following constructive
counterpart to the results from \cite{Joh,Carette}.

\begin{prop}\label{prop:3-neg}
  An equal input matrix\/ $M_c \in \cM^{}_3$ with parameter triple\/
  $(c^{}_1 , c^{}_2 , c^{}_3 )$ subject to\/ $c^{}_{i} > 0$ and\/
  $c > 1$ is embeddable if and only if\/
  $c \leqslant 1 + \ee^{-\Delta^{\min}}$ holds with\/ $\Delta^{\min}$
  from Eq.~\eqref{eq:Delta}. In the extremal case, one has\/
  $M_c = \exp (Q^{}_{\pm})$ with the two generators from
  Eq.~\eqref{eq:gen-ei}. Further, one has\/
  $\Delta^{\min} \geqslant \pi\sqrt{3}$ and thus\/
  $c^{\max}\leqslant 1+\ee^{-\pi\sqrt{3}}$.
\end{prop}

\begin{proof}
  For the maximal point on the ray, the claim follows from our above
  constructive calculation, once we show that this really corresponds
  to a minimum (if $w>0$) or a maximum (if $w<0$). The Hessian can be
  calculated, and reads
\[
    H \, = \, \frac{\pm \ts  \sqrt{\frac{c^{}_{1_{\vphantom{\chi}}} 
      \nts c^{\vphantom{I}}_2 c^{}_3}
      {c^{\vphantom{a}}_1 \nts + c^{}_2 + c^{}_3}}}
      {2 \pi \ts (c^{}_1 {+} c^{}_2)(c^{}_1 {+} c^{}_3)
      (c^{}_2 {+} c^{}_3)} 
      \begin{pmatrix}
      (c^{}_2 + c^{}_3)^2 & c^{}_1 c^{}_2 - c \ts c^{}_3 &
      c^{}_1 c^{}_3 - c \ts c^{}_2 \\ c^{}_1 c^{}_2 - c \ts c^{}_3 &
      (c^{}_1 + c^{}_3 )^2 & c^{}_2 c^{}_3 - c \ts c^{}_1 \\
      c^{}_1 c^{}_3 - c \ts c^{}_2 & c^{}_2 c^{}_3 - c \ts c^{}_1 &
      (c^{}_1 + c^{}_2)^2   \end{pmatrix} 
\]   
for the two signs of $w$. It is positive (resp.\ negative)
semi-definite, with one zero eigenvalue and two positive (resp.\
negative) ones.  The neutral direction is $(1,1,1)^{\mathsf{T}}$, so
we have a valley of minima (maxima) along this direction, relative to
the two remaining degrees of freedom, as $H$ does not change along the
valley. Both signs of $w$ give the same value for $\Delta^{\min}$, as
claimed.
  
It remains to see that every point on our ray with
$1 < c < c^{\ts \max} = 1 + \ee^{-\Delta^{\min}}$ is also
embeddable. But this is now evident, because we can replace the above
(optimal) values for $x,y,z$ by $x+\theta, y+\theta, z+\theta$ with
any $\theta>0$ without compromising the generator condition.  Since
$\Delta$ is linear in $x,y,z$, we can thus obtain any value of
$\Delta$ in the interval $[\Delta^{\min}, \infty)$, and the claimed
embeddability follows.

Finally, consider $\Delta^{\min}$ from Eq.~\eqref{eq:Delta}, with
$c^{}_{1} \geqslant c^{}_{2} \geqslant c^{}_{3} >0$. It is a simple
exercise to check that it reaches its minimum whenever the three $c_i$
are equal, which then is the value we saw in Example~\ref{ex:max-c},
with the corresponding value for $c^{\max}$ as stated.
\end{proof}

\begin{remark}\label{rem:alt-reason}
  By construction, the generators $Q^{}_{\pm}$ from
  Eq.~\eqref{eq:gen-ei} commute with the matrix
  $C = C(c^{}_{1}, c^{}_{2}, c^{}_{3})$ for the fixed parameter
  triple, so also $[Q^{}_{\pm}, Q^{}_{C}]=\nix$ for the rate matrix
  $Q^{}_{C} = C - c\, \one$. Consequently, for $\tau\geqslant 0$, we
  obtain
\[
  \exp (Q^{}_{\pm} + \tau \ts Q^{}_{C}) \, = \,
  \exp (\tau \ts Q^{}_{C}) \exp (Q^{}_{\pm}) \, = \,
  \exp (\tau \ts Q^{}_{C}) M_{c^{\max}} \ts ,
\]
which is an equal-input matrix with summatory parameter
\[
  c(\tau) = 1 + \ee^{-\tau \ts c^{\max}} (c^{\max}-1) \ts .
\]
With $\tau \geqslant 0$, we thus reach all values in $(1, c^{\max}]$,
and always get two embeddings.

Now, this was based on the choice $k=0$, and one can now repeat the
exercise for any $k\in\NN$ in Eq.~\eqref{eq:w-k}. For smaller and
smaller values of $c>1$, one thus obtains further embeddings, where
two new embeddings occur at
\[
    \Delta \, = \,  \frac{(2k+1) \ts \pi \ts \kappa \sqrt{c\ts}}
    {\sqrt{c^{\vphantom{a}}_{1} \nts c^{}_2 \ts c^{}_3}} 
\]
for each $k\in\NN$. We leave further details to the interested
reader.  \exend
\end{remark}

Let us now turn our attention to the generic situation for $d=3$.

\subsection{Cyclic cases}

When $\deg (q_{_M}) = 3$, we know that $M$ is cyclic. If $M$ is
embeddable, no eigenvalue can lie on the (closed) negative real axis,
by an application of Culver's criterion (Fact~\ref{fact:Culver}).
Here, the spectrum of $M$ can be real or not, which we now consider
separately.

When $\sigma (M) \subset \RR$, we are in the situation of
Theorem~\ref{thm:cyclic-1}, so all eigenvalues must be positive for
potential embeddability. Then, the spectral radius of $A=M\nts - \one$
satisfies $\varrho_{\nts_A} < 1$, and it remains to formulate an
effective condition for when the well-defined matrix
$Q = \log (\one + A)$ from \eqref{eq:log-def} is a generator. In this
situation, by Theorem~\ref{thm:Frob} and Corollary~\ref{coro:alg}, we
know that $\log (\one + A) = \alpha A + \beta A^2$ with
$\alpha, \beta \in \RR$. Now, there are still two cases.

When $A$ is diagonalisable, we have $\sigma (A) = \{ 0, \mu, \nu \}$
with $\mu , \nu \in (-1,0)$ and $\mu \ne \nu$. Then, it follows from
the SMT (see \eqref{eq:SMT} and \cite[Eq.~(4.6)]{BS2} for details) 
that 
\begin{equation}\label{eq:albe-1}
  \alpha \, = \, \frac{\mu^2 \log (1+\nu) -
      \nu^2 \log (1+\mu)}{\mu\ts\nu (\mu-\nu)}
  \quad \text{and} \quad
  \beta \, = \, \frac{-\mu \log (1+\nu) +
      \nu \log (1+\mu)}{\mu\ts\nu (\mu-\nu)} \ts ,
\end{equation}
which is the unique solution to the SMT-induced equation
\[
    \begin{pmatrix} \mu & \mu^2 \\ \nu & \nu^2 \end{pmatrix}
    \begin{pmatrix} \alpha \\ \beta \end{pmatrix} \, = \, 
    \begin{pmatrix} \log (1+\mu) \\ \log (1+\nu) \end{pmatrix} .
\]

Otherwise, when $A$ is cyclic but fails to be diagonalisable, its JNF
must be $0 \oplus \JJ^{}_{2} (\mu)$ with $\mu \in (-1,0)$, and
\cite[Eq.~(4.7)]{BS2} then gives
\begin{equation}\label{eq:albe-2}
 \alpha \, = \, \frac{2\ts \log (1+\mu)}{\mu} - \frac{1}{1+\mu}
  \quad \text{and} \quad
  \beta \, = \, \frac{1}{\mu (1 + \mu)} -
   \frac{\log (1+\mu)}{\mu^2} \ts ,
\end{equation}
which also emerges from \eqref{eq:albe-1} as a limit of de l'Hospital
type. This gives the following result.

\begin{prop}\label{prop:3-real}
  Let\/ $M\in \cM^{}_{3}$ by cyclic, with real spectrum. Then, $M$ is
  embeddable if and only if the following two conditions are
  satisfied, where\/ $A=M\nts - \one$ as before.
\begin{enumerate}\itemsep=2pt
\item All eigenvalues of\/ $M$ are positive, which automatically
  implies\/ $\varrho_{\nts_A} <1$.
\item The real matrix\/ $Q=\log (\one + A)$ is a generator, where\/
  $Q=\alpha A + \beta A^2$ with\/ $\alpha, \beta$ from 
  Eqs.~\eqref{eq:albe-1} or \eqref{eq:albe-2},
  depending on whether\/ $M$ is diagonalisable or not.
\end{enumerate}

\noindent
In this case, the embedding is unique.  \qed
\end{prop}

Otherwise, we have $\sigma (M) = \{ 1, \lambda, \cc{\lambda} \}$ with
$\lambda \in \CC \setminus \RR$, but we can still use
Corollary~\ref{coro:alg} for the generator $Q$. Here,
$\one \ne M=\ee^Q$ with $Q\ne \nix$, which implies $\tr (Q) <0$ and
thus $\lvert \lambda\rvert^2 = \det (M) = \ee^{\tr (Q)} < 1$. Then,
$z^{}_{0} \defeq \log (\lambda)$ via the standard branch of the
complex logarithm is well defined, and
$z^{}_{k} = z^{}_{0} + k \, 2 \pi \ii$ with $k\in\ZZ$ runs through
\emph{all} complex logarithms of $\lambda$. Now, setting
$\lambda = 1+\mu$ and writing $Q=\alpha A + \beta A^2$, the SMT 
leads to the conditions
\[
  \begin{pmatrix} \exp \bigl(\alpha \mu + \beta \mu^2 \bigr) \\
    \exp \bigl( \alpha \mubar + \beta {\mubar}^2 \bigr)\end{pmatrix}
  \, = \,  \begin{pmatrix} \lambda \\ \cc{\lambda} \end{pmatrix} ,
\]
for which we need a \emph{real} solution $\alpha,\beta$. Upon taking
logarithms, one can quickly check that this is only possible if we use
a complex-conjugate pair of them for the two lines, hence
\[
  \begin{pmatrix} \mu \;\; \mu^2 \\ \mubar \;\; \mubar^2 \end{pmatrix}
  \begin{pmatrix} \alpha \\ \beta \end{pmatrix}
  \, = \, \begin{pmatrix} z^{}_{k} \\ \cc{z^{}_{k}} \end{pmatrix}
\]
for some $k\in\ZZ$. This is uniquely solved by
\begin{equation}\label{eq:3-complex}
  \begin{pmatrix} \alpha^{}_{k} \\ \beta^{}_{k} \end{pmatrix}
  \, = \, \myfrac{1}{\lvert \mu \rvert^2 \ts (\mubar - \mu)} \,
  \begin{pmatrix} \mubar^2 z^{}_{k} - \mu^2 \ts \cc{z^{}_{k}} \\
    -\mubar \ts z^{}_{k} + \mu \ts \cc{z^{}_{k}} \end{pmatrix} ,
\end{equation}
which is indeed real. So, for each choice of $k\in\ZZ$, our matrix $M$
has precisely \emph{one} real logarithm, namely
$\alpha^{}_{k} A + \beta^{}_{k} A^2 \in \Mat (3, \RR)$. We can
summarise this as follows.

\begin{table}[t]
  \caption{Summary of embedding cases for $d=3$, where we use MG as
    abbreviation for the property to be a Markov
    generator. \label{tab:d3}}
\begin{center}
\renewcommand{\arraystretch}{1.2}    
\begin{tabular}{|c|c|c|c|c|} \hline
    $\deg (q_{_M})$ & JNF & condition(s)  & unique & details \\ \hline
  $1$ & $\one^{}_{3^{}_{\vphantom{\chi}}} = \exp (\nix^{}_{3}) $
       &  {\rev --- }  & yes & Fact~\ref{fact:one} \\ \hline
  $2$ & $  \begin{array}{c}\diag (1,1,\lambda) \\ \lambda \ne 1
           \end{array}$ & $\lambda \in (0,1)$ & yes &
           Lemma~\ref{lem:3-deg-2}{\ts}(1) \\ \cline{2-5}
         &  $\begin{array}{c} \diag (1,\lambda,\lambda) \\ \lambda > 0
         \end{array}$ & $\lambda \in (0,1)$ & $\begin{array}{c}
         \text{not} \\  \text{always} \end{array}$ &
         Lemma~\ref{lem:3-deg-2}{\ts}(2) \\ \cline{2-5} 
        &  $\begin{array}{c} \diag (1,\lambda,\lambda) \\ \lambda < 0
            \end{array}$ & $\lambda \in [-\ee^{-\pi\sqrt{3}}, 0)$
        & no & Prop.~\ref{prop:3-neg} \\ \hline
  $3$ &  $\begin{array}{c} \diag (1,\lambda_1,\lambda_2) \\
            1 \ne \lambda_1 \ne \lambda_2 \ne 1 \end{array}$
          & $\begin{array}{c} \lambda_1, \lambda_2 \in (0,1) \\
               \log (\one {+} A) \text{ is MG} \end{array}$ & yes
          & Prop.~\ref{prop:3-real} \\ \cline{2-5}
          & $\begin{array}{c} 1 \oplus \JJ_2 (\lambda) \\
            \lambda \ne 1 \end{array}$
          & $\begin{array}{c} \lambda \in (0,1) \\
               \log (\one {+} A) \text{ is MG} \end{array}$ & yes
          & Prop.~\ref{prop:3-real} \\ \cline{2-5}
          & $\begin{array}{c} \diag(1, \lambda, \overline{\lambda}\, 
            )\\ \lambda \in \CC \setminus \RR \end{array}$
          & $\begin{array}{c} \lvert \lambda \rvert \in (0,1) \\
             \text{$R_k$ is MG for some $k$} \end{array}$ & $\begin{array}{c}
         \text{not} \\  \text{always} \end{array}$
          & Prop.~\ref{prop:3-complex} \\ \hline
\end{tabular}
\end{center}
\end{table}

\begin{prop}\label{prop:3-complex}
  Let\/ $M\in\cM^{}_{3}$ have spectrum\/
  $\sigma (M) = \{1, \lambda, \overline{\lambda} \ts \}$ with\/
  $\lambda \in \CC\setminus \RR$.  Then, $M$ is embeddable if and only
  if the following two conditions are satisfied, with\/
  $A=M\nts - \ts \one$.
\begin{enumerate}\itemsep=2pt
\item One has\/ $0 < \lvert \lambda \rvert < 1$.
\item There is a\/ $k\in\ZZ$ such that the real logarithm\/
  $R^{}_{k} =\alpha^{}_{k} A + \beta^{}_{k} A^2$ of\/ $M\nts$, with\/
  $\alpha^{}_{k}$ and\/ $\beta^{}_{k}$ from Eq.~\eqref{eq:3-complex},
  is a generator, which then gives\/ $M=\ee^{Q}$ with\/ $Q=R^{}_{k}$.
\end{enumerate}
In this case, $R^{}_{k}$ is a generator for only finitely many\/
$k\in\ZZ$, and no other candidates exist. The number of solutions
is bounded by the integer\/ 
$\big\lfloor 1 - \frac{\log(\det(M))}{2 \pi \sqrt{3}}\big\rfloor$.
\end{prop}

\begin{proof}
  The first condition was derived in the paragraph after
  Proposition~\ref{prop:3-real}.
  
  The second condition follows from our above calculation, which also
  constructs all possible candidates for the real logarithms of $M$.
  That only finitely many of them can be generators follows once again
  from the bounds on the imaginary parts of their eigenvalues, as is
  detailed in \cite[Thm.~4.5]{CFR}, including the bound as stated.
\end{proof}

Multiple embeddings do occur, though not for sufficiently large values
of $\det (M)$. Also, Theorem~\ref{thm:cuth} can often be used to
exclude non-uniqueness. Explicit examples with multiple generators are
discussed in \cite{Speak}, which considers an equal-input example in
our terminology, and in \cite[Ex.~17]{Davies}, which fits
Proposition~\ref{prop:3-complex}; see also \cite[Rem.~5.5]{BS}, and
\cite{Cuth73} for a more general discussion of multiple embeddings.
The basic results of this section on $d=3$ are summarised in
Table~\ref{tab:d3}.

\section{Embedding in four dimensions}\label{sec:four}

If $M \in \cM_4$ has $\deg (q_{_M}) = 1$, we are back to $M=\one$ and
Fact~\ref{fact:one}, hence to $\one = \exp (\nix)$.  So, we need to
consider the cases $\deg (q_{_M}) \in \{ 2,3,4\}$.

\subsection{Cases of degree 2}

In view of Fact~\ref{fact:one-is-special}, there are three possible
cases here, all of which are diagonalisable, again because their
minimal polynomial has no repeated factor. First, let us look at
$M\in\cM_4$ with JNF $\diag (1 , 1 , 1 , \lambda)$ and
$\lambda \ne 1$, where $0 < \det (M) \leqslant 1$ then forces
$\lambda \in (0,1)$.  This class of matrices is erroneously claimed to
be impossible in \cite[Lemma~5.1]{CFR}, as can be seen from looking at
\begin{equation}\label{eq:possible}
      M \, = \, \begin{pmatrix} 1 & 0 & 0 & 0 \\ 0 & 1 & 0 & 0 \\
        0 & \; 0 \; & 1{-}a & a \\ 0 & 0 & b & 1{-}b \end{pmatrix}
      \, = \, \one^{}_{2} \oplus \begin{pmatrix}
        1{-}a & a \\ b & 1{-}b \end{pmatrix}
\end{equation}
with $a,b \in [0,1]$. Due to the block structure, one can invoke
Theorem~\ref{thm:Kendall} to see that $M$ is embeddable if and only if
$0 \leqslant a + b < 1$. More generally, in complete analogy to the
corresponding case in three dimensions from Section~\ref{sec:d-3.2},
the most general form of $M$ is
\[
     M \, = \, \one + (\lambda - 1) \bigl( \alpha^{}_{i} a^{}_{j}
     \bigr)_{1\leqslant i,j \leqslant 4}
\]
subject to the conditions $\sum_{i=1}^{4} \alpha^{}_{i} a^{}_{i} = 1$
and $\sum_{i=1}^{4} a^{}_{i} = 0$, and to the restrictions in sign and
absolute value to guarantee that $M$ is Markov. Beyond the matrix in
\eqref{eq:possible}, and its obvious relatives that emerge from
simultaneous permutations of rows and columns, one can also get
\begin{equation}\label{eq:possible-2}
  M \, = \, \begin{pmatrix} \lambda & x & y & z \\
    0 & 1 & 0 & 0 \\ 0 & 0 & 1 & 0 \\ 0 & 0 & 0 & 1 \end{pmatrix}
\end{equation}
with $\lambda=1-x-y-z$ and $x,y,z\geqslant 0$ subject to the condition
that $0 < x+y+z < 1$, so that $M$ is invertible and
$\lambda \in (0,1)$. As before, all other cases of this kind are
obtained by simultaneous permutations of rows and columns.  To see
that these are all possibilities (with $\lambda \ne 1$), one uses the
constraints on the $a_i$ and $\alpha_j$ to show that at most one
$\alpha_i < 0$ and at most one $\alpha_j > 0$. Then, a simple case
distinction, which we leave to the interested reader, gives the above
characterisation of this class.

In any case of this class, with $\lambda \in (0,1)$, $A = M\nts -\one$
is similar to $\diag (0,0,0,\lambda\ts {-}1)$, hence
$A^2 = (\lambda - 1) A$.  This means that $\log (\one + A)$ is once
again given by Eq.~\eqref{eq:simplification}.  Since this is a
generator, $M$ is embeddable. As $M=\ee^R$ implies that $R$ is
diagonalisable, there are also real logarithms with spectrum
$\sigma (R) = \{ 0, \pm\ts 2\pi \ii \ts k, \log (\lambda) \}$ for
$k\in\NN$. Among them, there can be further generators (but at most
finitely many), and the embedding need not be unique.

An analogous situation is met for $M$ with JNF
$\diag (1 , \lambda , \lambda , \lambda)$ and $\lambda \ne 1$. As in
the three-dimensional case before Eq.~\eqref{eq:eq-input-3}, via a
completely analogous calculation, one can check that the most general
Markov matrices of this type are the equal-input matrices
\[
  M \, = \, M_c \, = \, (1-c)\ts \one + C(c^{}_{1}, \ldots , c^{}_{4})
\]
with parameters $c^{}_{i} \geqslant 0$ and
$c = 1-\lambda = c^{}_{1} + \ldots + c^{}_{4} > 0$; see also
\cite{CFR}.

Here, we must have $\lambda \in (0,1)$ for embeddability, again due to
the determinant condition $0 < \det(M) \leqslant 1$ (and also already
by Fact~\ref{fact:Culver}), hence $c=1-\lambda\in(0,1)$ as well.  Once
again, $A^2 = (\lambda - 1)A$ and $\log (1+A)$ from
\eqref{eq:simplification} is a generator, so $M_c$ is embeddable;
compare also with \cite[Prop.~5.17]{CFR}. Other real logarithms of
$M_c$ exist, then with
$\sigma (R) = \{ 0, \log (\lambda), \log(\lambda) \pm 2\pi \ii\ts k
\}$ for $k\in\NN$, some of which might be generators.  So far, we have
the following.

\begin{lemma}\label{lem:4-deg-2-easy}
  Let the JNF of\/ $M\in\cM_4$ be\/ $\diag (1, 1, 1, \lambda)$ or\/
  $\diag (1, \lambda, \lambda, \lambda)$, with\/
  $\lambda < 1$. Then, we have\/ $\deg (q_{_M}) = 2$, and the
  following conditions are equivalent.
\begin{enumerate}\itemsep=2pt
\item The matrix\/ $M$ is embeddable.
\item One has\/ $\det (M) > 0$, which is equivalent to\/
  $\lambda \in (0,1)$.
\item  The matrix\/ $Q=\log (\one + A)$ is a generator. 
\end{enumerate}  
In this case, we get\/ $Q = \frac{-\log (\lambda)}{1-\lambda} \ts A$,
but the embedding need not be unique.  \qed
\end{lemma}

Let us note that there is a connection of the second case 
with Remark~\ref{rem:square} via taking square roots, which is
another instance of the subtle complications that may show up.
\smallskip

It remains to consider the JNF $\diag (1, 1, \lambda , \lambda)$,
where $\det (M)>0$ no longer implies $\lambda$ to be positive. If
$\lambda >0$, we get embeddability with
$Q = - \frac{\log (\lambda)}{1-\lambda} \ts A$ as in
Lemma~\ref{lem:4-deg-2-easy}. Also here, we have further candidates,
but any generator $Q$ with $M=\ee^Q$ must have $0$ as an eigenvalue
and thus spectrum
$\sigma (Q) = \{ 0, 0, \log (\lambda) \pm \ts 2 \pi \ii \ts k \}$ for
some $k\in\NN$, where at most finitely many $k$ can give a solution.
Since we consider real matrices, we need to work with the matrices of
the form
$\diag (0, 0) \oplus \bigl( \log (\lambda) \one^{}_{2} + 2 \pi k
I_{x,y,z}\bigr)$, as in Corollary~\ref{coro:real-log}{\ts}(1).

Otherwise, when $\lambda<0$, we must have $\lambda > -1$ due to
Elfving's condition (Fact~\ref{fact:Elf}).  In fact, as we shall see,
embeddability in this case is impossible if
$\lvert\lambda\rvert > \ee^{-\pi}$. Here, we need
Corollary~\mbox{\ref{coro:real-log}{\ts}(2)}. Putting the pieces
together and applying \cite[Lemma~3.1]{CFR}, one gets the following
result.

\begin{lemma}\label{lem:4-deg-2-mess}
  Let\/ $M\in\cM_4$ have\/ $\deg (q_{_M})=2$ and JNF
  $\one^{}_{2} \oplus \lambda \one^{}_{2}$, with\/
  $0 \ne \lambda \in [-1,1)$.  Then,
  $M = T^{-1} \diag (1, 1, \lambda, \lambda) \ts T$ for some\/
  $T \in \GL (4, \RR)$. If\/ $\lambda > 0$, $M$ is always embeddable,
  via the generator\/ $Q = - \frac{\log (\lambda)}{1-\lambda} \ts A$
  with\/ $A=M\nts - \one$. Further embeddings exist if and only if
\[
    R^{}_{k} (x,y,z) \, = \,
    T^{-1} \bigl( \diag (0, 0) \oplus
    \bigl( \log (\lambda) \one^{}_{2} + (2 k) \pi I_{x,y,z} 
    \bigr) \bigr) T
\]
is a generator for some\/ $0\ne k\in\ZZ$ and some\/ $x,y,z \in \RR$
with\/ $yz-x^2=1$ and\/ $z>0$.
  
Otherwise, if\/ $\lambda < 0$, the matrix\/ $M$ is embeddable if and
only if
\[
    R^{}_{k} (x,y,z) \, = \,
    T^{-1}  \bigl( \diag (0, 0) \oplus
    \bigl( \log  \ts\lvert \lambda \rvert \ts\one^{}_{2} + 
    (2k+1)  \pi I_{x,y,z} \bigr) \bigr) T
\] 
is a generator for some\/ $k\in\ZZ$ and some\/ $x,y,z \in \RR$ with\/
$yz-x^2=1$ and\/ $z>0$.

In both cases, at most finitely many solutions can exist, and the
range of\/ $k$ can be restricted via\/
$2 \pi \lvert k \rvert \leqslant \lvert\ts \log(\lambda)\rvert$ in the
first and via\/
$\lvert 2 k + 1\rvert \ts \pi \leqslant \big\lvert \! \log \lvert
\lambda \rvert \big\rvert$ in the second case, where the latter
excludes\/ $\lambda < - \ee^{-\pi}$.  \qed
\end{lemma}

Note that the appearance of the matrices $I_{x,y,z}$ comes from the
freedom in choosing $T$. Since $\diag (1, 1, \lambda, \lambda)$
certainly commutes with any matrix of the form $\one^{}_{2} \oplus B$
with $B\in\GL (2,\RR)$, a proper choice of $B$ would mean to simply
replace $I_{x,y,z}$ by $I = I^{}_{0,1,1}$. However, since it is not
obvious how to achieve this from the start, we need to formulate the
result as stated.  Once the `correct' $B$ is chosen, it becomes
transparent why only finitely many candidates exist.

\subsection{Cases of degree 3}

Taking into account Fact~\ref{fact:one-is-special}, there are four
possible cases with real spectrum, and one with a complex conjugate
pair. First, consider $M$ with $\sigma (M)\subset \RR$ and JNF
$\diag (1, 1, \lambda_1 , \lambda_2)$, then with
$1 \ne \lambda_1 \ne \lambda_2 \ne 1$, which means
$\lambda_1, \lambda_2 \in (0,1)$ for embeddability by
Fact~\ref{fact:Culver}. Then, by Lemma~\ref{lem:extend-unique}, only
one real logarithm with zero row sums exist, and this is
$\log (\one + A)$, so $M$ is embeddable if and only if this is a
generator.

Next, let us consider $M$ with JNF
$\one^{}_{2} \oplus \JJ_2 (\lambda)$ and $\lambda\ne 1$, then
necessarily with $\lambda \in (0,1)$ for embeddability. Employing
\cite[Thm.~1.27]{Higham}, we see that any real logarithm of $M$ must
be (complex) similar to
$\diag ( 2 \pi \ii \ts k, - 2 \pi \ii \ts k) \oplus
\left( \begin{smallmatrix} \log (\lambda) & 1/\lambda \\ 0 & \log
    (\lambda) \end{smallmatrix} \right)$ for some $k\in\ZZ$, where
\begin{equation}\label{eq:Jordan-gen}
     \exp \begin{pmatrix}  \log (\lambda) & 1/\lambda \\ 0 &
     \log (\lambda) \end{pmatrix} \, = \, \begin{pmatrix} \lambda & 1 \\
     0 & \lambda \end{pmatrix} , 
\end{equation}
as follows from a simple calculation. Again, only $k=0$ is possible
for a generator, in line with Lemma~\ref{lem:extend-unique}, and $M$
and $Q$ then have the same centraliser.  This once again means that
the principal matrix logarithm gives the only candidate, and we have
the following.

\begin{lemma}\label{lem:4-deg-3-real}
  Let\/ $M\in\cM_4$ have minimal polynomial of degree\/ $3$ and JNF
  $\one^{}_{2} \oplus \diag( \lambda_1 , \lambda_2)$, with\/
  $1 \ne \lambda_1 \ne \lambda_2 \ne 1$, or JNF\/
  $\one^{}_{2} \oplus \JJ_2 (\lambda)$, with\/ $\lambda \ne 1$.  Then,
  $M=\one+A$ is embeddable if and only if the following two conditions
  are satisfied.
\begin{enumerate}\itemsep=2pt
\item One has\/ $\lambda_1, \lambda_2 \in (0,1)$,
  respectively\/ $\lambda \in (0,1)$, which implies\/
  $\varrho_{_{\nts A}} < 1$.
\item The matrix\/ $\log (\one+A)$ is a generator.
\end{enumerate}
  In this case, the embedding is unique.
  
  Concretely, setting\/ $\lambda_1 = 1+\mu$ and\/ $\lambda_2=1+\nu$ in
  the first case, and\/ $\lambda = 1+\mu$ in the second, the logarithm
  can be calculated via the SMT from Eq.~\eqref{eq:SMT} as
\[
     \log (\one + A) \, = \, \alpha A + \beta A^2 ,
\]  
with the coefficients from Eq.~\eqref{eq:albe-1} in the first and from
Eq.~\eqref{eq:albe-2} in the second case.  \qed
\end{lemma}

Next is the case with JNF $\diag (1 , \lambda) \oplus \JJ_2 (\lambda)$
and $\lambda \ne 1$, which also has $\deg (q_{_M})=3$. Again, it can
only be embeddable if $\lambda \in (0,1)$ by Fact~\ref{fact:Culver}.
Then, we have $\varrho_{\nts_A} < 1$, and $\log (\one + A)$ is well
defined and the only real logarithm of $M$. It is similar to
$\diag \bigl(0, \log(\lambda)\bigr) \oplus \left( \begin{smallmatrix}
    \log (\lambda) & 1/\lambda \\ 0 & \log (\lambda) \end{smallmatrix}
\right)$, compare Eq.~\eqref{eq:Jordan-gen}, and we have the
following.

\begin{lemma}\label{lem:4-deg-3-mixed-J}
  Let\/ $M\in\cM_4$ have JNF\/
  $\diag (1,\lambda) \oplus \JJ_2 (\lambda)$ with\/ $\lambda \ne
  1$. Then, $\deg (q_{_M}) = 3$ and\/ $M$ is embeddable if and only if
  the following two conditions are satisfied.
\begin{enumerate}\itemsep=2pt
\item One has\/ $\lambda \in (0,1)$, which implies\/
     $\varrho_{_{\nts A}} < 1$.
\item The matrix\/ $\log (\one+A)$ is a generator.
\end{enumerate}
  In this case, the embedding is unique.
  Further, with\/ $\lambda =  1 + \mu$, one has\/
  $\log (\one + A) = \alpha A + \beta A^2$ with\/ $\alpha$ and\/
  $\beta$ as in  Eq.~\eqref{eq:albe-2}. \qed
\end{lemma}

Note that the matrix $M$ here is not cyclic, but we still get a unique
real logarithm.  This case can also be seen as a limit of the cyclic
matrices from Proposition~\ref{prop:4-cyclic-2} below.

For $\sigma (M)\subset \RR$, it remains to consider the case that $M$
has JNF $\diag (1, \lambda_1 , \lambda_2 , \lambda_2)$ with
$1 \ne \lambda_1 \ne \lambda_2 \ne 1$. Embeddability forces
$\lambda_1 \in (0,1)$, but $\lambda_2$ can be positive or negative.
Let us first look at $\lambda_2 > 0$. Then, $A$ has spectral radius
$\varrho_{\nts_A} < 1$ and $\log (\one + A)$ is well defined, but need
not be a generator, or even if it is, it need not be the only one.
So, we need again the matrices $I_{x,y,z}$ from
Fact~\ref{fact:real-log-2} to proceed.

Any generator $Q$ with $M=\ee^Q$ must have spectrum
$\sigma (Q) = \{ 0, \log(\lambda_1), \log (\lambda_2) \pm \ts 2 \pi
\ii \ts k \}$ for some $k\in\NN_0$, where finitely many $k\ne 0$ can
give a solution even if $\log (\one + A)$ fails.  The case
$\lambda_2<0$ is similar, except that $\log (\one+ A)$ does not
converge as a series.  Corollary~\ref{coro:real-log} gives us the form
of the real matrices we need to work with, as for
Lemma~\ref{lem:4-deg-2-mess}, and the two cases $\lambda_2>0$ and
$\lambda_2<0$ have to be treated separately as follows, invoking again
\cite[Lemma~3.1]{CFR}.

\begin{lemma}\label{lem:4-deg-3-mixed}
  Let\/ $M\in\cM_4$ have\/ $\deg (q_{_M})=3$ and JNF
  $\diag (1, \lambda_1 , \lambda_2 , \lambda_2)$, together with\/
  $1 \ne \lambda_1 \ne \lambda_2 \ne 1$.  Then,
  $M = T^{-1} \diag (1, \lambda_1, \lambda_2, \lambda_2) \ts T$ for
  some\/ $T \in \GL (4, \RR)$. Now, if all eigenvalues are positive,
  $M$ is embeddable if and only if
\[  
    R^{}_{k} (x,y,z) \, = \,
    T^{-1}  \bigl( \diag ( 0, \log (\lambda_1) )  \oplus
    ( \log (\lambda_2) \one^{}_{2} + (2 k) \pi  I_{x,y,z} 
    ) \bigr)  T
\]
is a generator for some\/ $k\in\ZZ$ and some\/ $x,y,z \in \RR$ with\/
$yz-x^2=1$ and\/ $z>0$, where\/ $k$ must also satisfy the condition\/
$2 \pi \lvert k \rvert \leqslant \lvert \ts \log(\lambda_2)\rvert$.
  
Likewise, if\/ $\lambda_1>0$ but\/ $\lambda_2<0$, the matrix\/ $M$ is
embeddable if and only if
\[  
    R^{}_{k} (x,y,z) \, = \,
    T^{-1} \bigl( \diag (0, \log (\lambda_1) ) \oplus
    ( \log  \ts\lvert \lambda_2 \rvert \ts\one^{}_{2} + 
    (2k\ts {+} 1)  \pi  I_{x,y,z} ) \bigr) T
\] 
is a generator for some\/ $k\in\ZZ$ and some\/ $x,y,z \in \RR$ with\/
$yz-x^2=1$ and\/ $z>0$. Here, $k$ is further restricted via\/
$\lvert 2 k + 1 \rvert \ts \pi \leqslant \big\lvert \!  \log \ts
\lvert \lambda_2\rvert \big\rvert$, which excludes\/
$\lambda_2 < - \ee^{-\pi}$.  \qed
\end{lemma}

Once again, by choosing $T$ properly, one can replace $I_{x,y,z}$ by
$I = I^{}_{0,1,1}$, which makes the structure a little more
transparent.

\begin{remark}\label{rem:square}
  A special case of Lemma~\ref{lem:4-deg-3-mixed} occurs when $M$ has
  JNF $\diag(1,\lambda,-\lambda,-\lambda)$ with $\lambda \in (0,1)$.
  If such an $M$ is embeddable, one has $M=\ee^Q$ where $Q$ must have
  spectrum $\sigma(Q) = \{0,\mu,\mu\pm (2k{+}1)\pi\ii\}$ with
  $\mu=\log(\lambda)$, for some $k\in\ZZ$. Then, we get
  $M^2 = \ee^{2 Q}$ with JNF
  $\diag (1, \lambda^2,\lambda^2,\lambda^2)$, which means that $M^2$
  must be an equal-input matrix; compare the discussion around
  Lemma~\ref{lem:4-deg-2-easy}.  Since $M^2$ is embeddable by
  construction, we know that its summatory parameter
  $c^{\ts\prime} = 1-\lambda^2$ must lie in $(0,1)$, as it does, and
  $M^2$ is also equal-input embeddable. Since $Q$ cannot be of
  equal-input type, we see here one Markov semigroup of the form
  $\{ \ee^{t \ts Q} : t \geqslant 0\}$ crossing another, which is the
  origin of multiple embeddings.  \exend
\end{remark}

Finally, we need to look at $M$ with JNF
$\diag (1, 1, \lambda, \overline{\lambda}\, )$ and
$\lambda \in \CC \setminus \RR$. As we must have $0$ as an eigenvalue
of any generator for $M$, we can almost repeat the arguments used for
Proposition~\ref{prop:3-complex}, invoke \cite[Thm.~4.5]{CFR} for
$d=4$, and then arrive at the following result,

\begin{lemma}\label{lem:4-deg-3-complex}
  Let the JNF of\/ $M\in\cM^{}_{4}$ be\/
  $\diag (1 , 1 , \lambda , \overline{\lambda})$, with\/
  $\lambda \in \CC\setminus \RR$.  Then, $\deg (q_{_M})=3$, and\/ $M$
  is embeddable if and only if the following two conditions are
  satisfied, with\/ $A=M\nts - \one$.
\begin{enumerate}\itemsep=2pt
\item One has\/ $0 < \lvert \lambda \rvert < 1$.
\item There is a\/ $k\in\ZZ$ such that the real logarithm\/
  $R^{}_{k} =\alpha^{}_{k} A + \beta^{}_{k} A^2$ of\/ $M$, with\/
  $\alpha^{}_{k}$ and\/ $\beta^{}_{k}$ from Eq.~\eqref{eq:3-complex},
  is a generator, which then gives\/ $M=\ee^{Q}$ with\/ $Q=R^{}_{k}$.
\end{enumerate}
In this case, $R^{}_{k}$ is a generator for only finitely many\/
$k\in\ZZ$, and no other candidates exist. The number of solutions is
bounded by the integer\/
$\big\lfloor 1 - \frac{\log(\det(M))}{2 \pi}\big\rfloor$.  \qed
\end{lemma}

Let us now turn to the generic situation for $d=4$.

\subsection{Cyclic cases}

A cyclic matrix $M\in\cM^{}_{4}$ can be diagonalisable, in which case
it is simple, or not, where it then has a non-trivial JNF. Each case
has two subcases to consider.  Common to all cases is that any
potential generator for the embedding must be a real logarithm of $M$
of the form $R = \alpha A + \beta A^2 + \gamma A^3$ with
$A = M \nts - \one$ and $\alpha, \beta, \gamma \in \RR$, by
Corollary~\ref{coro:alg}.  \smallskip

When $M \in \cM^{}_{4}$ is simple and potentially embeddable, there
are two cases to consider, namely
$\sigma (M) = \{ 1, \lambda^{}_{1}, \lambda^{}_{2}, \lambda^{}_{3} \}$
with distinct $\lambda_i \in (0, 1)$, again by Fact~\ref{fact:Culver},
and otherwise
$\sigma (M) = \{ 1, \lambda, \vartheta, \cc{\vartheta} \}$ with
$\lambda \in (0,1)$ and $\vartheta \in \CC\setminus\RR$ together with
$\lvert \vartheta \rvert < 1$, where $\lvert \vartheta \rvert = 1$ is
excluded by a result due to Elfving; see Fact~\ref{fact:Elf}.

If $\sigma (M)\subset \RR$, we are back to
Theorem~\ref{thm:cyclic-1}. Indeed, with $\lambda_i = 1 +\mu_i$, the
SMT from Eq.~\eqref{eq:SMT} gives the candidate
$R = \alpha A + \beta A^2 + \gamma A^3$ with
\[
  \begin{pmatrix} \mu^{}_{1} & \mu^{2}_{1} & \mu^{3}_{1} \\
    \mu^{}_{2} & \mu^{2}_{2} & \mu^{3}_{2} \\
    \mu^{}_{3} & \mu^{2}_{3} & \mu^{3}_{3} \end{pmatrix}
  \begin{pmatrix} \alpha \\ \beta \\ \gamma \end{pmatrix}
  \, = \, \begin{pmatrix} \log (\lambda^{}_{1}) \\
    \log (\lambda^{}_{2}) \\ \log (\lambda^{}_{3}) \end{pmatrix} .
\]
This leads to the unique, real solution
\begin{equation}\label{eq:4-real}
   \begin{pmatrix} \alpha \\ \beta \\ \gamma \end{pmatrix}
   \, = \, \begin{pmatrix} \frac{\mu^{}_{2} \mu^{}_{3}}{m^{}_{1}} &
     \frac{\mu^{}_{1} \mu^{}_{3}}{m^{}_{2}} &
     \frac{\mu^{}_{1} \mu^{}_{2}}{m^{}_{3}} \\ \rule[-8pt]{0pt}{22pt}
     - \frac{\mu^{}_{2} + \mu^{}_{3}}{m^{}_{1}} &
     - \frac{\mu^{}_{1} + \mu^{}_{3}}{m^{}_{2}} &
     - \frac{\mu^{}_{1} + \mu^{}_{2}}{m^{}_{3}} \\
     \frac{1}{m^{}_{1}} & \frac{1}{m^{}_{2}} & \frac{1}{m^{}_{3}}
     \end{pmatrix} \begin{pmatrix} \log (\lambda^{}_{1}) \\
    \log (\lambda^{}_{2}) \\ \log (\lambda^{}_{3}) \end{pmatrix}
\end{equation}
with $m_i = \mu_i \prod_{j\ne i} (\mu_{j} - \mu_{i})$. Here, on the
right-hand side, one always has to use the real logarithm, as no other
choice results in a real solution. So, we have the following result.

\begin{prop}\label{prop:4-real}
  Let\/ $M\in \cM^{}_{4}$ be simple, with real spectrum. Then, $M$ is
  embeddable if and only if the following two conditions are
  satisfied, where\/ $A=M\nts - \one$ as before.
\begin{enumerate}\itemsep=2pt
\item All eigenvalues of\/ $M$ are positive, which automatically
  implies\/ $\varrho_{\nts_A} <1$.
\item The real logarithm\/ $R=\log (\one + A)$ is a generator, where\/
  $R=\alpha A + \beta A^2 + \gamma A^3$ with the coefficients\/
  $\alpha, \beta, \gamma$ from Eq.~\eqref{eq:4-real}.
\end{enumerate}

\noindent
In this case, the embedding is unique.  \qed
\end{prop}

Next, consider the case that
$\sigma (M) = \{ 1, \lambda, \vartheta, \cc{\vartheta} \}$, where
embeddability forces $\lambda \in (0,1)$ by Culver's criterion
(Fact~\ref{fact:Culver}). We must also have
$\lvert \vartheta\rvert < 1$, again by Fact~\ref{fact:Elf}, where we
may assume $\imag (\vartheta) > 0$.  This implies
$\sigma (A) = \{ 0, \mu, \nu, \cc{\nu} \}$ with $\mu\in (-1,0)$ and
suitable conditions for $\nu$.  Since $\vartheta \not\in\RR$, the SMT
implies the equation
\[
 \begin{pmatrix} \mu & \mu^{2} & \mu^{3} \\
    \nu & \nu^{2} & \nu^{3} \\
    \cc{\nu} & \cc{\nu}^{2} & \cc{\nu}^{3} \end{pmatrix}
  \begin{pmatrix} \alpha \\ \beta \\ \gamma \end{pmatrix}
  \, = \, \begin{pmatrix} \log (\lambda) \\
    \log (\vartheta) + k \, 2 \pi \ii \\
    \log (\cc{\vartheta}) - k \, 2 \pi \ii \end{pmatrix}
\]
for some $k\in\ZZ$, where $\log$ is again the standard branch of the
complex logarithm. This is the most general case where
$\alpha, \beta, \gamma$ are real. Given $k$, the unique solution is
\begin{equation}\label{eq:4-simple-2}
  \begin{pmatrix} \alpha \\ \beta \\ \gamma \end{pmatrix}
  \, = \, \begin{pmatrix} \frac{\lvert \nu \rvert^2}{m^{}_{1}} &
    \frac{\mu\ts \cc{\nu}}{m^{}_{2}} &
    \frac{\mu\ts \nu}{m^{}_{3}} \\ \rule[-8pt]{0pt}{22pt}
    -\frac{\nu+\cc{\nu}}{m^{}_{1}} &
    -\frac{\mu+\cc{\nu}}{m^{}_{2}} & -\frac{\mu+\nu}{m^{}_{3}} \\
     \frac{1}{m^{}_{1}} & \frac{1}{m^{}_{2}} & \frac{1}{m^{}_{3}}
  \end{pmatrix}  \begin{pmatrix} \log (\lambda) \\
    \log (\vartheta) + k \, 2 \pi \ii \\
    \log (\cc{\vartheta}) - k \, 2 \pi \ii \end{pmatrix},
\end{equation}
with $m^{}_{1} = \mu (\mu-\nu)(\mu-\cc{\nu})$, which is real,
$m^{}_{2} = \nu (\nu-\mu)(\nu-\cc{\nu})$ and
$m^{}_{3} = \cc{\nu} (\cc{\nu}-\mu)(\cc{\nu}-\nu)$, which form a
complex-conjugate pair. In particular, for each $k\in\ZZ$, we get
precisely \emph{one} real logarithm of $M$ this way, which need not be
a generator though; bounds on $k$ follow again from
\cite[Thm.~4.5]{CFR}. The remainder of the arguments runs in complete
analogy to previous ones, and gives the following result.

\begin{prop}\label{prop:4-non-real}
  Let\/ $M\in \cM^{}_{4}$ be simple, with one complex-conjugate pair
  of eigenvalues. Then, $M$ is embeddable if and only if the following
  two conditions are satisfied.
\begin{enumerate}\itemsep=2pt
\item One has\/
  $\sigma (M) = \{ 1, \lambda, \vartheta, \cc{\vartheta} \}$ with\/
  $\lambda \in (0,1)$ and\/ $0<\lvert \vartheta \rvert < 1$.
\item One of the real logarithms of $M\nts$, as given by\/
  $R_{k}=\alpha A + \beta A^2 + \gamma A^3$ for some\/ $k\in\ZZ$
  with\/ $\alpha, \beta, \gamma$ from Eq.~\eqref{eq:4-simple-2}, is a
  generator.
\end{enumerate}
Here, at most finitely many $k\in\ZZ$ can lead to a generator, and no
further candidates exist. The number of solutions is bounded by the
integer\/ $\big\lfloor 1 - \frac{\log(\det(M))}{2 \pi}\big\rfloor$.
\qed
\end{prop}

It remains to consider the two cases where $M$ is cyclic, but not
diagonalisable. First, we have $\sigma (M) = \{ 1, \lambda \}$ with
$\lambda \in \RR$ and $\lambda \ne 1$, where embeddability forces
$\lambda \in (0,1)$ by Fact~\ref{fact:Culver}, and the JNF of $M$ must
then be $1 \oplus \JJ^{}_{3} (\lambda)$ by
Fact~\ref{fact:one-is-special}.  The real logarithm of $M$ is unique
by Lemma~\ref{lem:extend-unique}, and it can easily be calculated as
follows. Setting $\lambda = 1 + \mu$, the results from
\cite[Thm.~5.3]{BS2}, as detailed in its proof, imply the confluent
Vandermonde-type condition
\[
  \begin{pmatrix} \mu & \mu^2 & \mu^3 \\ 1 & 2 \mu & 3 \mu^2 \\
    0 & 2 & 6 \mu \end{pmatrix}
  \begin{pmatrix} \alpha \\ \beta \\ \gamma
  \end{pmatrix} \, = \, \begin{pmatrix} \log (\lambda) \\  
  \rule[-4pt]{0pt}{16pt} \lambda^{-1}
  \\ - \lambda^{-2} \end{pmatrix}
\]
with the unique solution
\begin{equation}\label{eq:4-cyclic-1}
  \begin{pmatrix} \alpha \\ \beta \\ \gamma \end{pmatrix} \, = \,
  \begin{pmatrix} \frac{3}{\mu} & -2 & \frac{\mu}{2} \\
   \rule[-8pt]{0pt}{20pt} -\frac{3}{\mu^2} & \frac{3}{\mu} & -1 \\
   \frac{1}{\mu^3} & -\frac{1}{\mu^2} 
   & \frac{1}{2 \mu} \end{pmatrix} 
   \begin{pmatrix} \log (\lambda) \\ 
    \rule[-4pt]{0pt}{16pt} \lambda^{-1}
  \\ - \lambda^{-2} \end{pmatrix} .
\end{equation}
This case can be summarised as follows.

\begin{prop}\label{prop:4-cyclic-1}
  A cyclic matrix\/ $M = \one + A\in\cM^{}_{4}$ with JNF\/
  $1 \oplus \JJ^{}_{3} (\lambda)$ is embeddable if and only if the
  following two conditions are satisfied.
\begin{enumerate}\itemsep=2pt
\item One has\/ $\lambda \in (0,1)$.
\item The real logarithm\/
  $\log (\one+A) = \alpha A + \beta A^2 + \gamma A^3$ with\/
  $\alpha, \beta, \gamma$ from Eq.~\eqref{eq:4-cyclic-1} is a
  generator.
\end{enumerate}  
In this case, the embedding is unique.  \qed
\end{prop}

\begin{table}[t]
\caption{Summary of embedding cases for $d=4$.\label{tab:d4}}
\begin{center}
\renewcommand{\arraystretch}{1.21}    
\begin{tabular}{|c|c|c|c|c|} \hline
    $\deg (q_{_M})$ & JNF & condition(s) & unique & details \\ \hline
  $1$ & $\one^{}_{4^{}_{\vphantom{\chi}}} = \exp (\nix^{}_{4})$ 
       &   {\rev --- }  & yes & Fact~\ref{fact:one} \\ \hline
  $2$ & $\begin{array}{c} \diag (1,1,1,\lambda) \\
           1 \ne \lambda \in [-1,1)
         \end{array}$ & $\lambda \in (0,1)$ &
                 $\begin{array}{c}
         \text{not} \\  \text{always} \end{array}$  &
           Lemma~\ref{lem:4-deg-2-easy} \\ \cline{2-5}
       &  $\begin{array}{c} \diag (1,\lambda,\lambda,\lambda)
                 \\ 1 \ne \lambda \in [-1,1)
           \end{array}$ & $\lambda \in (0,1)$ &
              $\begin{array}{c}
         \text{not} \\  \text{always} \end{array}$  &
         Lemma~\ref{lem:4-deg-2-easy} \\ \cline{2-5} 
       &  $\begin{array}{c} \diag (1,1,\lambda,\lambda) \\
             \lambda > 0 \end{array}$ & $ \lambda \in (0,1)$
       & $\begin{array}{c}
         \text{not} \\  \text{always} \end{array}$ &
         Lemma~\ref{lem:4-deg-2-mess} \\ \cline{2-5}
       &  $\begin{array}{c} \diag (1,1,\lambda,\lambda) \\
          \lambda < 0 \end{array}$ & $ \text{some } R_k (x,y,z) 
          \text{ is MG}$
       & $\begin{array}{c}
         \text{not} \\  \text{always} \end{array}$ &
         Lemma~\ref{lem:4-deg-2-mess} \\ \hline
  $3$ &  $\begin{array}{c} \diag (1,1,\lambda_1,\lambda_2) \\
            1 \ne \lambda_1 \ne \lambda_2 \ne 1 \end{array}$
          & $\begin{array}{c} \lambda_1, \lambda_2 \in (0,1) \\
               \log (\one {+} A) \text{ is MG} \end{array}$ & yes
          & Lemma~\ref{lem:4-deg-3-real} \\ \cline{2-5}
          & $\begin{array}{c} \one_2 \oplus \JJ_2 (\lambda) \\
            \lambda \ne 1 \end{array}$
          & $\begin{array}{c} \lambda \in (0,1) \\
               \log (\one {+} A) \text{ is MG} \end{array}$ & yes
          & Lemma~\ref{lem:4-deg-3-real} \\ \cline{2-5}
          & $\begin{array}{c} \diag(1,\lambda) \oplus
             \JJ_2 (\lambda) \\  \lambda \ne 1 \end{array}$
          &  $\begin{array}{c} \lambda \in (0,1) \\
               \log (\one {+} A) \text{ is MG} \end{array}$ & yes
          & Lemma~\ref{lem:4-deg-3-mixed-J} \\ \cline{2-5}
          & $\begin{array}{c} \diag (1,\lambda_1,\lambda_2,\lambda_2) \\
            1 \ne \lambda_1 \ne \lambda_2 \ne 1 \end{array}$ 
          & $\begin{array}{c} \lambda_1, \lambda_2 \in (0,1) \\
              \text{some $R_k (x,y,z)$ is MG} \end{array}$ 
          & $\begin{array}{c} \text{not} \\ \text{always}
               \end{array}$
          & Lemma~\ref{lem:4-deg-3-mixed} \\ \cline{3-5}
          & & $\begin{array}{c} \lambda_1 \in (0,1), \; \lambda_2 < 0 \\
               \text{some $R_k (x,y,z)$ is MG} \end{array}$ 
          & $\begin{array}{c} \text{not} \\ \text{always}
               \end{array}$
          & Lemma~\ref{lem:4-deg-3-mixed} \\ \cline{2-5}
          & $\begin{array}{c} \diag(1,1,\lambda,\overline{\lambda}\, )
              \\ \lambda \in \CC \setminus \RR \end{array}$
          & $\begin{array}{c} \lvert \lambda \rvert \in (0,1) \\
             \text{some $R_k$ is MG} \end{array}$ & $\begin{array}{c}
         \text{not} \\  \text{always} \end{array}$
          & Lemma~\ref{lem:4-deg-3-complex} \\ \hline
  $4$ & $\begin{array}{c} \diag (1,\lambda_1,\lambda_2,\lambda_3) \\
        \text{simple spectrum} \end{array}$ & $ \begin{array}{c}
        \lambda_1, \lambda_2, \lambda_3 \in (0,1) \\
        \log (\one {+} A) \text{ is MG} \end{array} $ & yes 
      & Prop.~\ref{prop:4-real} \\ \cline{2-5}
      & $\begin{array}{c} \diag (1,\lambda,\vartheta,
        \overline{\vartheta}\, ) \\ 1 \ne \lambda \in [-1,1), \,
           \vartheta \in \CC \setminus \RR \end{array}$
      & $ \begin{array}{c} \lambda \in (0,1), \;
                   \lvert \vartheta \rvert \in (0,1) \\
                   \text{some $R_k$ is MG} \end{array} $
      & $\begin{array}{c} \text{not} \\ \text{always}
               \end{array}$
      & Prop.~\ref{prop:4-non-real} \\ \cline{2-5}
      & $ \begin{array}{c} 1 \oplus \JJ_3 (\lambda) \\
        1 \ne \lambda \in [-1,1) \end{array} $                     
      & $ \begin{array}{c} \lambda \in (0,1), \;
                   \lvert \vartheta \rvert \in (0,1) \\
              \log (\one{+}A) \text{ is MG} \end{array} $
      & yes & Prop.~\ref{prop:4-cyclic-1} \\ \cline{2-5}
      & $ \begin{array}{c} \diag (1,\lambda_1) \oplus \JJ_2 
        (\lambda_2) \\ 1 \ne \lambda_1 \ne \lambda_2 \ne 1
        \end{array} $
      & $ \begin{array}{c} \lambda_1, \lambda_2 \in (0,1) \\
        \log (\one {+} A) \text{ is MG} \end{array} $
      & yes & Prop.~\ref{prop:4-cyclic-2} \\ \hline
\end{tabular}
\end{center}
\end{table}

The remaining case is
$\sigma (M) = \{ 1, \lambda^{}_{1}, \lambda^{}_{2}\} \subset \RR$,
where $1\ne \lambda^{}_{1} \ne \lambda^{}_{2} \ne 1$ and, without loss
of generality, $M$ has JNF
$\diag (1 , \lambda^{}_{1}) \oplus \JJ^{}_{2} (\lambda^{}_{2})$. By
Fact~\ref{fact:Culver}, embeddability is at most possible for
$\lambda^{}_{1}, \lambda^{}_{2} \in (0,1)$, and the real logarithm of
$M$ is unique by Lemma~\ref{lem:extend-unique}. Here, the results from
\cite[Thm.~5.3]{BS2} give
\[
  \begin{pmatrix} \mu^{}_{1} & \mu^{2}_{1} & \mu^{3}_{1} \\
   \rule[-9pt]{0pt}{23pt} \mu^{}_{2} & \mu^{2}_{2} & \mu^{3}_{2} \\
   1 & 2 \mu^{}_{2} & 3 \mu^{2}_{2} \end{pmatrix}
   \begin{pmatrix} \alpha \\ \beta \\ \gamma \end{pmatrix} \, = \,
   \begin{pmatrix} \log (\lambda^{}_{1}) \\  
   \nts\rule[-6pt]{0pt}{16pt} \log (\lambda^{}_{2})\ts
    \\  \lambda^{-1}_{2} \end{pmatrix} ,
\]
with the unique solution
\begin{equation}\label{eq:4-cyclic-2}
 \begin{pmatrix} \alpha \\ \beta \\ \gamma \end{pmatrix} \, = \,
  \begin{pmatrix} 
  \frac{\mu^{2}_{2}}{\mu^{}_{1} (\mu^{}_{1} \nts - \mu^{}_{2})^2} & 
  \frac{\mu^{}_{1} (2 \mu^{}_{1} \nts - \ts 3 \mu^{}_{2})}
     {\mu^{}_{2} \ts ( \mu^{}_{1} \nts - \mu^{}_{2})^2} & 
  \frac{-\mu^{}_{1}}{\mu^{}_{1} \nts - \mu^{}_{2}} \\
   \rule[-10pt]{0pt}{26pt} \frac{-2 \mu^{}_{2}}{\mu^{}_{1} 
   (\mu^{}_{1} \nts - \mu^{}_{2})^2} & 
  \frac{3 \mu^{2}_{2} - \mu^{2}_{1}}
     {\mu^{2}_{2} \ts ( \mu^{}_{1} \nts - \mu^{}_{2})^2} & 
  \frac{\mu^{}_{1} + \ts \mu^{}_{2}} {\mu^{}_{2} 
    (\mu^{}_{1} \nts - \mu^{}_{2})} \\
   \frac{1}{\mu^{}_{1} (\mu^{}_{1} \nts - \mu^{}_{2})^2} & 
   \frac{\mu^{}_{1} \nts - 2 \mu^{}_{2}}
   {\mu^{2}_{2} \ts ( \mu^{}_{1} \nts - \mu^{}_{2})^2} 
   & \frac{-1}{\mu^{}_{2} (\mu^{}_{1} \nts - \mu^{}_{2})} \end{pmatrix} 
   \begin{pmatrix} \log (\lambda^{}_{1}) \\  
    \nts\rule[-6pt]{0pt}{16pt}\log(\lambda^{}_{2})\ts
  \\  \lambda^{-1}_{2} \end{pmatrix} .
\end{equation}
This implies the following result.

\begin{prop}\label{prop:4-cyclic-2}
  Let\/ $M\in\cM^{}_{4}$ be cyclic with\/
  $\sigma (M) = \{ 1, \lambda^{}_{1}, \lambda^{}_{2} \}$, for\/
  $1 \ne \lambda_1 \ne \lambda_2 \ne 1$, and JNF\/
  $\diag (1 , \lambda^{}_{1}) \oplus \JJ^{}_{2} (\lambda^{}_{2})$.
  Then, $M = \one + A$ is embeddable if and only if the following two
  conditions hold.
\begin{enumerate}\itemsep=2pt
\item One has\/ $\lambda^{}_{1}, \lambda^{}_{2} \in (0,1)$,
  with\/ $\lambda^{}_{1} \ne \lambda^{}_{2}$.
\item The matrix\/ $\alpha A + \beta A^2 + \gamma A^3$
   with\/  $\alpha, \beta, \gamma$ from
   Eq.~\eqref{eq:4-cyclic-2} is a generator.
\end{enumerate}  
In this case, the embedding is unique.  \qed
\end{prop}

The basic results of this section are summarised in
Table~\ref{tab:d4}. In combination with \cite{CFR}, this can be turned
into an algorithmic approach to the embedding problem for $d=4$.

\section{Application to phylogenetics}\label{sec:appl}

In this section, we will briefly discuss some models that are in use
for the nucleotide mutation schemes in molecular evolution. We shall
always use the ordering $( A, G, C, T )$.

First, and perhaps simplest, let us look at the widely used
\emph{equal-input} model \cite{Steel,BS2}. The Markov matrices of
equal-input type have the form
\begin{equation}\label{eq:eq-inp-4}
    M_c \, = \, (1-c) \ts \one + C (c^{}_{1}, \ldots , c^{}_{4} ) \ts ,
\end{equation}
where the $C$-matrix contains four equal rows of the form
$(c^{}_{1}, \ldots , c^{}_{4} )$, with parameters $c_i \geqslant 0$
and summatory parameter $c = c^{}_{1} + \ldots + c^{}_{4}$.  For $M_c$
to be Markov, we also need $c \leqslant 1 + c_i$ for all $i$, which
further implies $0 \leqslant c \leqslant \frac{4}{3}$.  Since $d=4$ is
even, \cite[Prop.~2.12]{BS2} or an application of
Lemma~\ref{lem:4-deg-2-easy} gives the following consequence.

\begin{coro}\label{coro:mut-EI}
  The four-dimensional equal-input Markov matrix\/ $M_c$ from
  Eq.~\eqref{eq:eq-inp-4} is embeddable if and only if its summatory
  parameter satisfies\/ $0\leqslant c < 1$. In this case, one
  embedding is\/ $M_c = \ee^Q$ with the equal-input generator
\[
       Q \, = \, \frac{-\log(1-c)}{c_{\vphantom{\chi}}} \ts A \ts ,
\]        
where\/ $A = M\nts - \one$ and\/ $Q=\nix$ for\/ $c=0$.  For\/ $c>0$,
there can be at most finitely many other embeddings, but none with an
equal-input generator.  \qed
\end{coro}

This result contains the constant-input matrices as the
special case
$c^{}_{1} = \ldots = c^{}_{4} = \frac{c}{4}$, which comprises the
Jukes--Cantor matrices from \cite{JC}. Unlike the situation of odd
dimension, compare Example~\ref{ex:max-c} and \cite{BS2}, no 
further embeddable case can occur here.  \smallskip

A mild extension of the equal-input class is provided by the
\emph{Tamura--Nei} (TN) model from \cite{TN}; see also \cite{CS}.
Here, one considers Markov matrices of the form
\begin{equation}\label{eq:TN-def}
  M \, = \, \begin{pmatrix} 
  * & a^{}_{2} \ts \kappa^{}_{1} & a^{}_{3} & a^{}_{4} \\
  a^{}_{1} \kappa^{}_{1} & * & a^{}_{3} & a^{}_{4} \\
  a^{}_{1} & a^{}_{2} & * & a^{}_{4} \kappa^{}_{2} \\
  a^{}_{1} & a^{}_{2} & a^{}_{3} \kappa^{}_{2} & * 
  \end{pmatrix} 
\end{equation}
with $a^{}_{i} \geqslant 0$ and $\kappa^{}_{j}\geqslant 0$ for all
$i,j$, subject to the condition that $M$ is Markov, which means that
the sum of the off-diagonal elements in each row must not exceed
$1$. The $*$ in each row is the unique number to ensure row sum $1$.
Similarly, when the row sums are all $0$, we shall speak of a TN
generator. The algebraic structure of TN generators is such that their
exponential is always a Markov matrix of TN type, see \cite{CS}, while
it is not clear that a real logarithm should preserve the structure.

TN matrices occur within the often-used hierarchy of
\emph{time-reversible} models, like those implemented in popular
computational phylogenetics software such as \cite{Darriba}.  While
there is no principal reason to restrict to time-reversible models,
this class is pretty versatile and seems general enough while having a
number of computational advantages; compare \cite{Fels}.  Time
reversibility also implies that these matrices have real spectrum, and
they include the HKY matrices from \cite{HKY} via
$\kappa^{}_{1} = \kappa^{}_{2}$.

$M$ from \eqref{eq:TN-def} is always diagonalisable, and has spectrum
$\sigma (M) = \{ 1, \lambda_1, \lambda_2, \lambda_3\} \subset \RR$ with
\[
\begin{split}
   \lambda^{}_{1} \, & = \, 1 - (a^{}_{1} + a^{}_{2} + a^{}_{3} + a^{}_{4})
   \ts ,  \\
   \lambda^{}_{2} \, & = \, 1 - \kappa^{}_{1} (a^{}_{1} + a^{}_{2}) 
         - (a^{}_{3} + a^{}_{4})  \ts , \\
   \lambda^{}_{3} \, & =  \, 1 - (a^{}_{1} + a^{}_{2}) - \kappa^{}_{2}
        (a^{}_{3} + a^{}_{4})   \ts .   
\end{split}        
\]
The spectrum is generically simple. However, the Markov condition does
not imply that all eigenvalues are positive. The $\lambda_i$ all lie
in $(0,1)$ if and only if
\begin{equation}\label{eq:TN-cond}
\begin{split}
   0 \, < & \, \min \{ 1, \kappa^{}_{1} \} (a^{}_{1} + a^{}_{2} ) +
    \min \{ 1, \kappa^{}_{2} \} (a^{}_{3} + a^{}_{4} )  \quad \text{and} \\
    &  \, \max \{ 1, \kappa^{}_{1} \} (a^{}_{1} + a^{}_{2} ) +
    \max \{ 1, \kappa^{}_{2} \} (a^{}_{3} + a^{}_{4} )   \, < \, 1 \ts ,
\end{split}    
\end{equation}
where the first condition ensures that $1$ is a simple eigenvalue.
This gives the following generic answer, by an application of
Proposition~\ref{prop:4-real}.

\begin{coro}\label{coro:TN-gen}
  Let\/ $M\in\cM_4$ be a TN matrix as in \eqref{eq:TN-def}, and assume
  that it is simple. Then, $M$ is embeddable if and only if the
  conditions in \eqref{eq:TN-cond} are satisfied. In this case, the
  embedding is unique, and the generator is of TN type.  \qed
\end{coro}

When $M$ fails to be simple, the spectrum is still real, but has one
or several degeneracies. If $\lambda_1=\lambda_2=\lambda_3$, we must
have $\kappa^{}_{1} =1$ or $a^{}_{1}=a^{}_{2}=0$ together with
$\kappa^{}_{2} =1$ or $a^{}_{3}=a^{}_{4}=0$. In any of these cases,
one is back to the equal-input matrices, with
$c=a^{}_{1} + \ldots + a^{}_{4}$, which are fully covered by
Corollary~\ref{coro:mut-EI}.

The remaining degenerate cases can lead to the JNF
$\one^{}_{2} \oplus \lambda \one^{}_{2}$ with $\lambda \in (0,1)$, for
instance via $\kappa^{}_{1}=a^{}_{3} = a^{}_{4}=0$, which is always
embeddable by Lemma~\ref{lem:4-deg-2-mess}, but possibly not in a
unique way. Finally, one can have the JNF
$\diag (1, \lambda) \oplus\lambda'\one^{}_{2}$ within the HKY class
(where $\kappa^{}_{1} = \kappa^{}_{2}$), with $1,\lambda,\lambda'$
distinct, then needing Lemma~\ref{lem:4-deg-3-mixed}. We leave further
details to the reader.  \smallskip

The \emph{Kimura} 3 ST model, or K3{\ts}S{\ts}T for short, was
introduced in \cite{Kim81} and comprises all Markov matrices of the
form
\begin{equation}\label{eq:K3ST}
   M \, = \, \begin{pmatrix} * & x & y & z \\
   x & * & z & y \\ y & z & * & x \\ z & y & x & * \end{pmatrix}
   \, = \, \bigl( 1 - (x+y+z)\bigr) \one + x K_1 + y K_2 + z K_3
\end{equation}
with parameters $x,y,z \geqslant 0$ subject to the condition
$x+y+z\leqslant 1$. In each row of $M$, the $*$ again stands for the
unique element that makes the row sum equal to $1$, and the definition
of the matrices $K_i$ is implicit. Under matrix multiplication, the
four mutually commuting matrices
$\{ \one, K_1, K_2, K_3 \} \simeq C_2 {\times}\ts C_2$ form Klein's
$4$-group, with $C_2$ denoting the cyclic group with two elements. The
class of K3{\ts}S{\ts}T \emph{generators} are the matrices of the form
\begin{equation}\label{eq:kst-gen}
    Q (x,y,z) \, = \, x K_1 + y K_2 + z K_3 - (x+y+z) \one \ts,
\end{equation}
which thus constitute an Abelian class under matrix multiplication.

Note also that, under the $( A, G, C, T)$-ordering, Kimura's two
parameter model from \cite{Kim80}, called K2{\ts}P, is the special
case where one takes $y=z$. It was discussed in detail, with various
surprising results, in \cite{CFR-K2}. Further, the set of matrices
that are both K3{\ts}S{\ts}T and equal input are precisely the
constant-input matrices.

The matrix $M$ in \eqref{eq:K3ST} is symmetric, hence always has real
spectrum, namely
\begin{equation}\label{eq:eigen-order}
    \sigma (M) \, = \, \{ 1, 1-2(x+z), 1-2(y+z), 1-2(x+y) \} 
    \, = \, \{ 1, \lambda_1, \lambda_2, \lambda_3 \} \ts ,
\end{equation}
which is meant as a multi-set if degeneracies occur.  Here,
$\deg (q_{_M})=1$ only for $M=\one$, while $\deg (q_{_M}) = 2$ occurs
for $x=y=z\ne 0$. This gives the constant-input matrices, which are
covered by Corollary~\ref{coro:mut-EI}. Next, when
$\deg (q_{_M}) = 3$, we have $x\ne y = z$, which is the K2{\ts}P model
covered in \cite{CFR-K2}, or a scheme that is equivalent to it via a
permutation from $S_4$.

Finally, in the generic case that $M$ in \eqref{eq:K3ST} is simple,
which is true if and only if the non-negative numbers $x,y,z$ are
distinct, we are in the situation of Proposition~\ref{prop:4-real}.
In particular, $M$ has a real logarithm if and only if all its
eigenvalues are positive. There is then only one candidate,
$\log (\one + A)$, which is a generator \cite{CFR-K3} if and only if
the three non-unit eigenvalues $\lambda_1, \lambda_2, \lambda_3$
satisfy the three inequalities
\begin{equation}\label{eq:products}
    \lambda_1 \, \geqslant \, \lambda_2 \ts \lambda_3 \, , \quad
    \lambda_2 \, \geqslant \, \lambda_1 \ts \lambda_3 \, , \quad
    \lambda_3 \, \geqslant \, \lambda_1 \ts \lambda_2 \ts .
\end{equation}
This can here be derived from Proposition~\ref{prop:4-real}, which
gives an explicit form of the only possible generator, or from an
explicit argument based on the diagonalisation of $M$ with the
involutory Fourier matrix
$W = \frac{1}{2} \left( \begin{smallmatrix} 1 & 1 \\ 1 &
    -1 \end{smallmatrix} \right) \otimes \left( \begin{smallmatrix} 1
    & 1 \\ 1 & -1
  \end{smallmatrix} \right)$ from the discrete Fourier transform over
$C_2 {\ts\times\ts} C_2$; see \cite{CFR-K3} for a previous derivation
and some details. The surprisingly simple set of inequalities emerges
from additive conditions on the logarithms of the eigenvalues, for the
generator property of $\log (\one + A)$, upon exponentiation.

\begin{coro}\label{coro:K3ST}
  Let the matrix\/ $M$ from \eqref{eq:K3ST} be simple, so\/
  $\sigma (M) = \{ 1, \lambda_1, \lambda_2, \lambda_3 \}$ with
  distinct\/ $\lambda_i$, all different from\/ $1$. Then, $M$ is
  embeddable if and only if all\/ $\lambda_i > 0$ together with
  obeying the inequalities in \eqref{eq:products}. In this case, with
  the\/ $\lambda_i$ in the order from \eqref{eq:eigen-order} and then
  setting
\[
    (s^{}_{1}, s^{}_{2}, s^{}_{3}) \, \defeq \,  s^{}_{1} \log(\lambda^{}_1) 
    + s^{}_{2} \log(\lambda^{}_2) + s^{}_{3} \log(\lambda^{}_3) 
    \quad \text{for } s^{}_{i} \in \{ +,-\} \ts ,
\]
    one has\/ $M=\ee^Q$ with\/ 
\[
      Q \, = \, \myfrac{1}{4} \begin{pmatrix}
      (+,+,+) & (-,+,-) & (+,-,-) & (-,-,+) \\
      (-,+,-) & (+,+,+) & (-,-,+) & (+,-,-) \\
      (+,-,-) & (-,-,+) & (+,+,+) & (-,+,-) \\
      (-,-,+) & (+,-,-) & (-,+,-) & (+,+,+) \end{pmatrix}
\]
which is a\/ \textnormal{K3{\ts}S{\ts}T} generator, and the embedding
is unique.  \qed
\end{coro}

There are of course many more $4$-dimensional models used in
phylogenetics where the embedding problem is relevant and should be
studied. For instance, the strand-symmetric model is considered in
\cite{CFR-strand}. One interesting result for this model is that the
authors identify an open set of embeddable matrices where the
principal matrix logarithm is not a generator.

\section{Extension to time-inhomogeneous cases}\label{sec:inhom}

The standard continuous-time Markov chain solves the \emph{ordinary
  differential equation} (ODE) $\dot{M} = M Q$ with a constant
generator (or rate matrix) $Q$; see \cite{A,W} for an introduction to
classic ODE theory.  More generally, in many real-world applications,
one has to admit \emph{time-dependent} generators, $Q (t)$, which
leads to the Cauchy (or initial value) problem
\[
    \dot{M} (t) \, = \, M(t) \ts Q(t) \quad \text{with } M(0) = \one \ts .
\]
If the generators commute, that is, if $[ Q(t) , Q(s) ] = \nix$ for
all $t,s \geqslant 0$, as for instance the \mbox{K3{\ts}S{\ts}T}
generators from Eq.~\eqref{eq:kst-gen} do, the solution is simply
given by
\begin{equation}\label{eq:standard-solution}
    M (t) \, = \, \exp \Bigl( \, \text{\small $\int_{0}^{t}$ }
    \! Q(\tau) \dd \tau \Bigr) ,
\end{equation}
and any Markov matrix that arises in such a solution is also
embeddable in the classic sense discussed above, because
$\int_{0}^{t} Q (\tau) \dd \tau$ is still a generator. However, this
might (and will) change once one also considers families of generators
that do not commute.

Let us begin by recalling a standard result from ODE theory; see
\cite[Ch.~III]{A} for background and further details. Concretely,
consider the linear ODE on $\RR^d$, with a row vector $x(t)$ in view
of our setting, given by
\begin{equation}\label{eq:vector}
     \dot{x} (t) \, = \, x(t) \ts Q (t) 
\end{equation}
with $ Q \colon [t^{}_0 , t] \xrightarrow{\quad} \Mat (d, \RR)$
continuous, where $t > t^{}_0$ is arbitrary, but fixed. We shall also
need the limiting case where we let $t \to \infty$.  If $X(t)$ denotes
the corresponding (left) fundamental system, and setting $t^{}_0 = 0$,
it satisfies the Cauchy problem
\begin{equation}\label{eq:CP}
  \dot{X} (t) \, = \, X(t) \ts Q (t)  \quad
  \text{with } X (0) \, = \, \one \ts .
\end{equation}
Invoking the transposed version of the classic \emph{Peano--Baker
  series} (PBS), see \cite{BSch} and references therein, one has the
following consequence of the Picard--Lindel\"{o}f theorem in
conjunction with the standard Picard iteration,\footnote{We refer to
  the \textsc{WikipediA} entry on the Picard--Lindel\"{o}f theorem for
  a summary, and to \cite{W} for details.}  as follows from
\cite[Thms.~1 and 2]{BSch} by matrix transposition.

\begin{prop}\label{prop:PBS}
  The Cauchy problem of Eq.~\eqref{eq:CP} with a continuous matrix
  function\/ $Q$ on\/ $\RR_{\geqslant 0}$ has a unique solution.  It
  can be represented by the P{\nts}BS
\[
    X (t) \, = \, \one + \sum_{n\geqslant 1} I_n (t)
\]   
with\/ $I^{}_1 (t) = \int_{0}^{t} Q (\tau) \dd \tau$ and the recursive
structure\/
$I_{n+1} (t) = \int_{0}^{t} I_n (\tau) \ts Q (\tau) \dd \tau$ for\/
$n \in \NN$. In particular, this series is compactly convergent in any
standard matrix norm.  \qed
\end{prop}

Note that the order under the integral is changed in comparison to
\cite{BSch}, which matches the changed order of matrices in
\eqref{eq:CP} and reflects the standard use of row sum normalisation
for Markov matrices in probability theory.  As explained in detail in
\cite{BSch}, the solution formula with the PBS reduces to the standard
one in Eq.~\eqref{eq:standard-solution} when the $Q (t)$ commute with
one another.

\begin{remark}
  Observe that $I^{}_2$ can be calculated in two different
  ways, namely
\[
    I^{}_2 (t) \, = \int_{0}^{t} \int_{0}^{t^{}_2} Q(t^{}_1)\ts Q(t^{}_2)
    \dd t^{}_1 \dd t^{}_2 \, = \int_{0}^{t} \int_{t^{}_2}^{t} Q(t^{}_2) \ts
    Q(t^{}_1) \dd t^{}_1 \dd t^{}_2 \ts ,
\]   
as follows from changing the order of integration together with a
change of variable transformation. Then, $I^{}_2$ can also be written
as
\[
    I^{}_2 (t) \, = \, \myfrac{1}{2}
    \int_{0}^{t} \int_{0}^{t} T [ Q(t^{}_1)\ts Q(t^{}_2)] \ts
    \dd t^{}_1 \dd t^{}_2
\]   
where $T$ denotes \emph{time ordering} according to
\[
    T [ Q(t^{}_1) Q(t^{}_2)]  \, \defeq \, \begin{cases}
    Q(t^{}_1) Q(t^{}_2) , & \text{if } t^{}_1 \leqslant t^{}_2 , \\
    Q(t^{}_2) Q(t^{}_1),  & \text{otherwise}. \end{cases}  
\]   
Similarly, if $T[Q(t^{}_1)\ts Q(t^{}_{2}) \cdots Q(t^{}_n)]$ denotes
the analogous time-ordered version according to
$t^{}_1 \leqslant t^{}_2 \leqslant \cdots \leqslant t^{}_n$, one finds
the alternative expression
\[
    I^{}_n (t) \, = \, \myfrac{1}{n!} \int_{0}^{t} \int_{0}^{t}
    \cdots \int_{0}^{t} T[ Q(t^{}_1) \ts Q(t^{}_{2}) \cdots Q(t^{}_n)]
    \ts \dd t^{}_1 \dd t^{}_2 \cdots \dd t^{}_n \ts ,
\]
which is often used in physics, then usually with the opposite
ordering due to the action of the matrices to the right (instead of to
the left as above). The PBS is then called the Dyson series, or the
\emph{time-ordered exponential}, though its actual calculation (if
possible at all) is a lot easier with the recursive formulation stated
in Proposition~\ref{prop:PBS}.  \exend
\end{remark}

\begin{theorem}\label{thm:Markov}
  Consider the Cauchy problem of Eq.~\eqref{eq:CP} under the
  assumption that\/ $Q(t)$ is continuous and a Markov generator for
  all\/ $t \geqslant 0$.  Then, the solution flow\/
  $\{ X (t) : t \geqslant 0 \}$ consists of Markov matrices only.
\end{theorem}

\begin{proof}
  This follows from standard results of ODE theory, as given in
  \cite[Sec.~16]{A}, so we only sketch the required steps.\footnote{To
    better appreciate the flow property of the solution, one should
    write $X(0,t)$ instead of $X(t)$, and more generally
    $X(t^{}_{0},t)$ for the solution on the time interval
    $[t^{}_{0}, t]$ with the fitting initial condition at $t^{}_{0}$,
    which then leads to the concatenation rule
    $X(t^{}_{0}, t^{}_{1} ) X(t^{}_{1}, t)= X(t^{}_{0}, t)$; see
    \cite{A} for more details.}  Consider Eq.~\eqref{eq:vector}. Under
  the assumption on $Q$, for every $x$ on the boundary of the closed
  cone $(\RR_{\geqslant 0})^d$, the vector $x \ts Q (t)$ has a
  direction that points inside the closed cone, because all
  off-diagonal elements of $Q (t)$ are non-negative and $x_i=0$ then
  implies $(x \ts Q (t))_i \geqslant 0$.  This means that all
  non-negative rescalings of $x \ts Q(t)$ remain in the cone and the
  flow thus cannot pass the boundary to the outside. By the transposed
  version of \cite[Thm.~16.9 and Cor.~16.10]{A}, which formulate the
  general result from \cite[Thm.~16.5]{A} for this special case,
  $(\RR_{\geqslant 0})^d$ is forward invariant under the flow.  Since
  this applies to any row of $X (t)$, with $X(0)=\one$, all entries of
  $X(t)$ for $t\geqslant 0$ are non-negative.

  Now, the row sums of all $I_n (t)$ in the PBS are zero, as can
  easily be checked inductively, because all operations stay within
  the (non-unital) algebra $\cA^{(d)}_{0}$ of real matrices with zero
  row sums. Since the first term of the PBS is $\one$, we see that the
  row sums of $X(t)$ are always $1$, which together with the above
  establishes the Markov property as claimed.
\end{proof}  

The fundamental system $X(t)$ for \eqref{eq:vector}, which solves the
matrix equation \eqref{eq:CP}, is the Markov matrix we are after,
called $M(t)$ from now on.  Let us add that
$\bs{1} \defeq (1, \ldots , 1)^{\mathsf{T}}$ is a right eigenvector of
$M(0) = \one$ with eigenvalue $1$. Since $Q(t)$ is a generator, we get
\[
   \frac{\dd}{\dd t} \bigl( M(t) \bs{1} \bigr) \, = \, \dot{M} (t) \bs{1}
   \, = \, M(t) \ts Q(t) \bs{1} \, = \, \bs{0} \ts .
\]
This shows that $\bs{1}$ is a right eigenvector of $M(t)$ for
\emph{all} $t\geqslant 0$, which is another way to see that the row
sums of $M(t)$ are always $1$.

\begin{coro}\label{coro:det}
  Let\/ $M (t) = X (t)$ be the solution from
  Theorem~\textnormal{\ref{thm:Markov}}, with\/ $Q (t)$ being a Markov
  generator for all\/ $t \geqslant 0$. Then, one has\/
  $0 < \det ( M (t)) \leqslant 1$ for all\/ $t \geqslant 0$.
\end{coro}

\begin{proof}
  Set $w (t) = \det (M (t))$. By Liouville's theorem \cite[Prop.~11.4
  and Cor.~11.5]{A}, which is sometimes also known as Abel's identity,
  we have $ \dot{w} (t) = \tr (Q (t)) \, w (t) $ for all
  $t\geqslant 0$, with $w(0) = 1$, and thus
\begin{equation}\label{eq:trace}
    w (t) \, = \, \exp \Bigl( \, \text{\small $\int_{0}^{t}$} \nts
         \tr (Q (\tau)) \dd \tau \Bigr) .
\end{equation}
Since, for all $\tau\geqslant 0$, the diagonal elements of $Q (\tau)$
are non-positive, also $\tr (Q(\tau))$ and hence the argument of the
exponential are non-positive, and the claim follows from the
properties of the exponential function.
\end{proof}

\begin{remark}\label{rem:Volterra}
  In general, one needs to go beyond continuous families
  $\{ Q(t) : t \geqslant 0 \}$.  In particular, one wants to include
  piecewise continuous functions, so that a jump from one generator to
  another is covered. This is achieved by simply replacing the Cauchy
  problem with the corresponding Volterra integral equation,
\[
  M(t) \, = \, \one + \! \int_{0}^{t} \! \!
  M (\tau) \ts Q(\tau) \dd \tau  ,
\]  
and using the solution theory accordingly. We skip further details of
this standard step, and refer to the literature for details
\cite{A,W}. As we shall see later, $Q(t)$ piecewise constant for
$t\geqslant 0$ will essentially be sufficient; see \cite{Good,FS} for
the full measure-theoretic treatment.  \exend
\end{remark}

The result of Theorem~\ref{thm:Markov} remains true if $Q(t)$ is
piecewise continuous, as one can then simply use the arguments in the
proof for each of the finitely many continuity intervals in time,
where the initial condition is always the last Markov matrix from the
previous interval. Beyond this, if $Q$ is locally Lebesgue integrable,
which means that each of its entries is a locally Lebesgue-integrable
function, one can approximate $Q (t)$ on any given compact time
interval by step functions, in the sense of standard Lebesgue theory,
and use a limit argument; see \cite[Ch.~VI.9]{Lang} for background.

\begin{coro}\label{coro:Markov}
  If\/ $Q(t)$ is a locally Lebesgue-integrable function of generators,
  and\/ $M(t)$ is a solution of the Volterra integral equation from
  Remark~\textnormal{\ref{rem:Volterra}}, this solution is Markov for
  all\/ $t\geqslant 0$, with\/ $M (0)=\one$.  \qed
\end{coro}

An important observation is that $M (t)$ is absolutely continuous with
bounded entries, so $M(t) \ts Q(t)$ is again locally integrable.  The
same type of observation applies to the iterative definition of the
$I_n (t) $ in the PBS, which implies that each of these integrals
defines an absolutely continuous matrix function, with
$I_n (t) \in \cA^{(d)}_{0}$ for all $n\in\NN$ and $t\geqslant 0$, as
one can show by induction. This has the following consequence.

\begin{coro}\label{coro:algebra}
  Let\/ $\cA$ be a subalgebra of\/ $\cA^{(d)}_{0}\!$, hence closed
  under addition and matrix multiplication, and assume that\/ $Q (t)$
  defines a locally Lebesgue-integrable matrix function of generators,
  with\/ $Q (t) \in \cA$ for all\/ $t\geqslant 0$.  Then, the PBS
  for\/ $M(t)$ defines an absolutely continuous function of Markov
  matrices that satisfy\/ $M (t) = \one + A(t)$ for all\/
  $t\geqslant 0$, where each\/ $A(t)$ is a rate matrix from\/ $\cA$.
\end{coro}

\begin{proof}
  The first integral, $I^{}_1 (t) = \int_{0}^{t} Q(\tau) \dd \tau$, is
  well defined and absolutely continuous, with $I^{}_1 (t) \in \cA$
  for each $t\geqslant 0$.  Then, the integrand for $I^{}_2$ is
  $I^{}_1 \ts Q$, which is Lebesgue-integrable on any interval of type
  $[0,t]$, and
  $I^{}_2 (t) = \int_{0}^{t} I^{}_1 (\tau)\ts Q(\tau) \dd \tau$ is
  again well defined and absolutely continuous. Since each
  $I^{}_2 (t)$ lies once again in $\cA$, we can argue inductively and
  obtain that $M (t) = \one + A(t)$, where
  $A(t) = \sum_{m=1}^{\infty} I_m (t)$ converges and lies in $\cA$ for
  each $t\geqslant 0$.

  Thus, we see that each $A(t)$ has zero row sums. Since $M(t)$ is
  Markov, by Corollary~\ref{coro:Markov}, we can conclude that $A(t)$
  is actually a rate matrix, for any $t\geqslant 0$.
\end{proof}

The significance of this statement is that, even under our
time-inhomogeneous scheme, the type of generator within a certain
class (as given by a subalgebra of $\cA^{(d)}_{0}\!$, say) forces the
same type of structure on the Markov matrices, as we have seen for the
equal-input matrices discussed in and around
Corollary~\ref{coro:mut-EI}. This can be seen as some kind of
consistency property, which needs the algebraic notion.

\begin{example}\label{ex:doubly-stoch}
Consider the two symmetric generators
\[
   Q^{}_{1} \, = \, \begin{pmatrix}
   -1 & 1 & 0 \\ 1 & -1 & 0 \\ 0 & 0 & 0 \end{pmatrix}
   \quad \text{and} \quad
   Q^{}_{2} \, = \, \begin{pmatrix}
   -1 & 0 & 1 \\ 0 & 0 & 0 \\ 1 & 0 & -1 \end{pmatrix} ,
\]
which lead to the Markov matrices
\[
    M^{}_{1} (t) \, = \, \ee^{t \ts  Q_1} \, = \, \begin{pmatrix}
    1{-}a & a & 0 \\ a & 1{-}a & 0 \\ 0 & 0 & 1 \end{pmatrix}
    \quad \text{and} \quad
    M^{}_{2} (s) \, = \, \ee^{s \ts Q_2} \, = \, \begin{pmatrix}
    1{-}b & 0 & b \\ 0 & 1 & 0 \\ b & 0 & 1{-}b \end{pmatrix}
\]
with $a=\frac{1}{2} (1 - \ee^{-2 \ts t})$ and $b = \frac{1}{2} 
(1 - \ee^{-2 s})$. They are also symmetric, while this is no
longer the case for the (doubly stochastic) product,
\begin{equation}\label{eq:bad-prod}
   M^{}_{1} (t) M^{}_{2} (s) \, = \, \begin{pmatrix}
   (1{-}a)(1{-}b) & a & (1{-}a) b \\ a (1{-}b) & 1{-}a & ab \\
   b & 0 & 1{-}b \end{pmatrix} .
\end{equation}
Indeed, the symmetric generators do \emph{not} form an algebra, but
sit inside the doubly stochastic generators, which do, and which thus
form the relevant class to consider in this context.  \exend
\end{example}

Note that the matrix in \eqref{eq:bad-prod} is not embeddable, due to
the $0$ in the last row, though it is a product of two embeddable
ones. Thus, we now extend the notion of embedabbility of Markov
matrices as follows.

\begin{definition}\label{def:gen}
  A Markov matrix $M\in\cM_d$ is called embeddable \emph{in the
    generalised sense}, or g-\emph{embeddable} for short, when it
  occurs as an $X(t)$ in the solution of Eq.~\eqref{eq:CP}, or in the
  corresponding Volterra integral equation from
  Remark~\ref{rem:Volterra}, with $\{ Q (t) : t \geqslant 0 \}$ being
  a locally Lebesgue-integrable family of Markov generators.
  
  The set of all g-embeddable Markov matrices in $d$ dimensions
  is denoted by $\gemb$.
\end{definition} 

The more general version with Lebesgue integrability of $Q (t)$ will
effectively be simplified to a piecewise continuous generator family
shortly. Clearly, an embeddable Markov matrix is also g-embeddable,
via considering a constant function $Q(t)\equiv Q$.  For $d=2$,
Corollary~\ref{coro:det} also gives us the following consequence.

\begin{coro}\label{coro:Kendall-inhom}
  If a Markov matrix\/ $M \in \cM_2$ is \textnormal{g}-embeddable, it
  is also embeddable, and the two notions of embeddability agree for\/
  $d=2$, so\/
  $\cM^{\ts\vphantom{g}\mathrm{e}}_{2} = \cM^{\mathrm{ge}}_{2}\!$.
\end{coro}

\begin{proof}
  Let $M \in \cM_2$ be embeddable in the generalised sense.  When
  $Q(t)$ is continuous, by Corollary~\ref{coro:det}, we then know that
  $\det (M) \in (0,1]$, which means it is embeddable by Kendall's
  criterion (Theorem~\ref{thm:Kendall}). The conclusion is easily
  extended to $Q$ locally Lebesgue-integrable, because $w (t)$ from
  Eq.~\eqref{eq:trace} still is the determinant of $M (t)$, which is
  absolutely continuous, hence $\det (M) \in (0,1]$ also in this case.
    
  The other direction is clear.
\end{proof}  

Let us next look at a simpler pair of generators than that of 
Example~\ref{ex:doubly-stoch}.

\begin{example}\label{ex:there-is-more}
Consider the two non-commuting elementary (or Poisson) generators
\[
    Q^{}_{1} \, = \, \begin{pmatrix} -1 & 1 & 0 \\ 0 & 0 & 0 \\
    0 & 0 & 0 \end{pmatrix} \quad \text{and} \quad
    Q^{}_{2} \, = \, \begin{pmatrix} 0 & 0 & 0 \\ 0 & 0 & 0 \\
    1 & 0 & -1 \end{pmatrix} 
     ,
\]
which give the exponentials
$\ee^{t \ts Q_i} = \one + ( 1 - \ee^{-t} ) Q_{i}$ for $ i \in \{
1,2\}$. One finds
\[
   M(t,s) \, \defeq \, \ee^{t \ts Q^{}_1}\ts \ee^{s \ts Q^{}_2} \, = \,
   \one + a Q^{}_{1} + b \ts Q^{}_{2} \, = \, 
   \begin{pmatrix} 1{-}a & a & 0 \\ 0 & 1 & 0 \\ b & 0 & 1{-}b
   \end{pmatrix}
\]
with $a = 1-\ee^{-t}$ and $b=1-\ee^{-s}$.  When $ab>0$, the matrix
$M=M(t,s)$ cannot be embeddable because $M_{31} M_{12} > 0$, but
$M_{32}=0$, which violates the transitivity property that is necessary
for embeddability; see \cite[Thm.~3.2.1]{Norris} or
\cite[Prop.~2.1]{BS}. By construction, each $M (t,s)$ is g-embeddable
though, where the time evolution uses a piecewise constant generator
family, first being $Q^{}_{2}$ until time $s$, where it switches to
$Q^{}_{1}$ and continues until time $t+s$.

Let us also check this via the results discussed earlier. For
$t,s > 0$ with $t\ne s$, the spectrum of $M$ is
$\{1, 1 \ts{-} a, 1 \ts{-} b \}$, hence simple and positive, wherefore
we know from Proposition~\ref{prop:3-real} that the only possibility
for an embedding of $M = \one + A$ comes from
\[
     \log (\one + A) \, = \, 
     \frac{a^2 \log (1-b) - b^2 \log (1-a)}{a b (b-a)} A + 
     \frac{a \log (1-b) - b \log (1-a)}{a b (b-a)} A^2
\]
with the coefficients calculated via Eq.~\eqref{eq:albe-1}. Since
$A^2 = - a^2 Q^{}_{1} - b^2 Q^{}_{2} + a b Z$ with
\[
    Z \, = \, Q^{}_{2} \ts Q^{}_{1} \, = \, \begin{pmatrix}
    0 & 0 & 0 \\ 0 & 0 & 0 \\ -1 & 1 & 0 \end{pmatrix},
\]
this gives
$\log (\one + A) \, = \, -\log (1-a) Q^{}_{1} - \log (1-b) Q^{}_{2} +
\gamma Z$ with
\[
\begin{split}
   \gamma \, & = \,   \frac{a \log (1-b) - b \log (1-a)}{b-a} \\[1mm]
   & = \, - ab \Bigl( \myfrac{1}{2} + \myfrac{1}{3} (a+b) +
   \myfrac{1}{4} (a^2 + ab + b^2) + \myfrac{1}{5} (a^3 + a^2 b +
   a b^2 + b^3) + \ldots \Bigr) 
\end{split}
\]
which is clearly negative for any $0 < a\ne b < 1$. This carries over
to the case $0<a=b$ by a limiting argument of de l'Hospital type,
where we note that $M(t,t)$ for $t>0$ is still cyclic, but not
diagonalisable.  We thus see that the only real logarithm of $M(t,s)$
for $ts>0$ has a negative entry at the second position in its third
row, and is thus not a generator.  \exend
\end{example}

Clearly, the same type of argument as in
Example~\ref{ex:there-is-more} can be used for any dimension $d>3$,
whence we can summarise our informal discussion as follows.

\begin{fact}
  For any\/ $d\geqslant 3$, there are non-embeddable matrices\/
  $M \nts \nts \in \nts \cM_d$ that are \textnormal{g}-embeddable,
  so\/ $\emb \subsetneq \gemb$ for any\/ $d\geqslant 3$.  \qed
\end{fact}

Let us state one algebraic advantage of $\gemb$ that was noticed and
used in \cite[Thm.~2.7]{Joh73}.

\begin{lemma}\label{lem:star}
  The set\/ $\gemb$ is a monoid under ordinary matrix multiplication,
  and it is star-shaped with respect to the singular constant-input
  matrix
\[
    J_d \, \defeq \, \myfrac{1}{d} \begin{pmatrix} 1 & \cdots & 1 \\
      \vdots & \vdots & \vdots \\ 1 & \cdots & 1 \end{pmatrix}
    \, = \, \myfrac{1}{d} \, C( 1, \ldots , 1) \ts ,
\]   
  which is an idempotent that lies on the boundary of\/ $\gemb$.
\end{lemma}

\begin{proof}
  The semigroup property of $\gemb$ is obvious from the definition,
  and $\one$ clearly is its neutral element.
  
  Now, consider $M_c \defeq (1-c) \one + c J_d$, which is a
  constant-input matrix and as such certainly embeddable (and hence in
  $\emb$) for all $0\leqslant c < 1$. For any fixed $P\in\gemb\!$, we
  know that
\[
    P M_c \, = \, (1-c) P + c J_d 
\]  
lies in $\gemb$ for every $c\in [0,1)$, where we have used that
$P J_d = J_d$ holds because $P$ is Markov. So, $\gemb$ is indeed
star-shaped from $J_d$, while all other claims are clear.
\end{proof}

It is possible to say more on special matrix classes, one being the
equal-input matrices, where one needs to distinguish even and odd
dimensions.

\begin{lemma}\label{lem:no-news}
  Let\/ $M\in\cM_d$ be an equal-input matrix. If\/ $d$ is even, $M$ is
  \textnormal{g}-embeddable if and only if it is also embeddable.
  For\/ $d$ odd, the time-inhomogeneous approach with all\/ $Q(t)$ of
  equal-input type does not produce new \textnormal{g}-embeddable
  cases.
\end{lemma}

\begin{proof}
  When $d$ is even and $M$ is g-embeddable, its determinant must lie
  in $(0,1]$ by Corollary~\ref{coro:det}. Since $M$ is equal input, we
  have $\det (M) = (1-c)^{d-1}$, which implies $c\in [0,1)$, and $M$
  is also embeddable, with an equal-input generator, by
  \cite[Prop.~2.12]{BS2}.

  When $d$ is odd, we know that other types of embeddability are
  possible for $c>1$, but then not with equal-input generators.  If
  $Q(t)$ is equal input for all $t\geqslant 0$, the solution of
  $\dot{M} = M\ts Q$ (respectively of the corresponding Volterra
  integral equation) will be equal input for all $t\geqslant 0$ if
  $M(0)$ is equal input, as follows from the PBS because the algebra
  $\cC^{(d)} \subsetneq \cA^{(d)}_{\ts 0}$ generated by the
  equal-input rate matrices is closed under addition and
  multiplication, so $M(t) = \one + A(t)$ with $A(t)\in\cC^{(d)}$ for
  all $t\geqslant 0$; compare Corollary~\ref{coro:algebra}.

  Now, assume that $M(0)$ has summatory parameter $c(0) < 1$, which is
  certainly true if $M(0)=\one$. Then, $c(t) < 1$ for all
  $t\geqslant 0$ because $c(t)=1$ would mean
  $\det\bigl( M(t)\bigr)=0$, which is impossible due to
  Corollary~\ref{coro:det}.  Consequently, each $M(t)$ will also be
  equal-input embeddable, again by \cite[Prop.~2.12]{BS2}, and no new
  cases emerge here.

  The only remaining case is that $M(0)$ is equal input and
  embeddable, but with $c(0)>1$. Then, by a standard calculation, one
  sees that $M(t)$ remains equal input, with
\[
  c(t) \, = \, c(0) \ts \exp
  \biggl( - \int_{0}^{t} \tilde{c} (\tau) \dd \tau \biggr)
\]
with $\tilde{c}(\tau)$ being the summatory parameter of
$Q(\tau)$. Clearly, $c(t)$ is non-increasing, and a minor deformation
argument around the approach we used in Remark~\ref{rem:alt-reason}
shows that $M (t)$ remains embeddable if $M (0)$ is.
\end{proof}

It is clear that a more systematic analysis of g-embeddability of
matrix models will be necessary, which we defer to future work.  At
this point, we take a look at the most general model, which is based
on $\cA^{(d)}_{0}$ and admits an approach via Poisson matrices. To
describe it, let $E_{ij} \in \Mat (d,\RR)$ be the elementary matrix
that has a single $1$ in position $(i,j)$ and $0\ts$s everywhere
else. They satisfy the multiplication rule
$E_{ij} E_{k\ell} \, = \, \delta_{jk} E_{i\ell}$.

A \emph{Poisson matrix} is any Markov matrix of the form
$M = \one - a E_{ii} + a E_{ij}$ for some $i\ne j$, then with
$a\in [0,1]$. Since $\det (M) = 1-a$, the matrix $M$ is singular if
and only if $a=1$. Likewise, any matrix of the form
$Q = -\alpha E_{ii} + \alpha E_{ij}$ with $i\ne j$ and
$\alpha \geqslant 0$ is a \emph{Poisson generator}, where one often
restricts to $\alpha > 0$ to exclude the trivial case $Q=\nix$;
compare Ex.~\ref{ex:there-is-more}.  A simple calculation shows that
$Q^2 = -\alpha \ts Q$ and then
\[
    \exp (Q) \, = \, \one + (1 - \ee^{-\alpha} ) R_{ij}
    \quad \text{with } \, R_{ij} \defeq - E_{ii} + E_{ij} \ts .
\]
So, a Poisson matrix, with parameter $a$, is embeddable if and only if
$0\leqslant a < 1$, which are precisely all non-singular cases.

\begin{example}
 Let $d=2$ and consider the product of two Poisson matrices,
\[
   \begin{pmatrix} 1{-}a & a \\ 0 & 1 \end{pmatrix}
   \begin{pmatrix} 1 & 0 \\ b & 1{-}b \end{pmatrix} \, = \,
   \begin{pmatrix} 1 {-} a\ts (1{-}b) & a\ts (1{-}b) \\
   b & 1{-}b \end{pmatrix}
\] 
for $a,b \in [0,1)$. Both are regular, and so is their product, with
determinant $1- (a+b-ab)$. One can check that
$0 \leqslant a+b-ab < 1$, which shows embeddability of the product by
Kendall's criterion from Theorem~\ref{thm:Kendall}.  Since $a+b-ab$
can take all values in $[0,1)$, forming the product of just two
Poisson matrices covers all embeddable cases, in line with
Corollary~\ref{coro:Kendall-inhom}.  \exend
\end{example}

Let us collect some general properties of g-embeddable matrices,
starting with the following necessary criterion from
\cite{Good,Fryd80a}.

\begin{lemma}\label{lem:g-emb-nec}
  If\/ $M\in\cM_d$ is \textnormal{g}-embeddable, it must satisfy 
\[
    \prod_{i=1}^{d} m^{}_{ii} \, \geqslant \, \det (M) \, > \, 0 \ts .
\]  
So, if\/ $m^{}_{ii} = 0$ for some\/ $i$, the matrix\/ $M$ is neither
embeddable nor \textnormal{g}-embeddable.  \qed
\end{lemma}

An astonishing general result is the following characterisation of
$\gemb$ due to Johansen.

\begin{theorem}[{\cite[Thms.~2.5 and 2.6]{Joh73}}]
  Every\/ $M\in\gemb$ can be approximated arbitrarily well by finite
  products of Poisson matrices, and every\/ $M\in\mathrm{int} (\gemb)$
  has a representation as a finite product of Poisson matrices, known
  as a Bang-Bang representation.
  
  Moreover, Markov idempotents and products of Markov idempotents lie
  on\/ $\partial \gemb\!$. They are finite products of regular or
  singular Poisson matrices.  \qed
\end{theorem}

The difficult matrices are thus the non-singular elements on the
boundary of $\gemb\!$, and no general characterisation seems to be
known. The key point of this result is that it is essentially
sufficient to consider the ODE from Eq.~\eqref{eq:CP} with families of
Markov generators that are piecewise constant, at least in an
approximate sense.

\subsection{Characterisations for \textit{d}\textbf{ = 3}}

The situation is a little better here. First, Lemma~\ref{lem:g-emb-nec} 
has some partial converse as follows.

\begin{prop}[{\cite[Thm.~3.1]{Fryd80a}}]
  Let\/ $M\in\cM_3$ satisfy\/ $m^{}_{ij}=0$ for some\/ $i\ne j$. Then,
  $M$ is \textnormal{g}-embeddable if and only if the double
  inequality of Lemma~\textnormal{\ref{lem:g-emb-nec}} holds. These
  matrices are precisely the ones that can be represented as products
  of at most\/ $5$ Poisson matrices.  \qed
\end{prop}

All remaining cases share the condition to be totally positive, so
$m^{}_{ij} > 0$ for all $1\leqslant i,j \leqslant 3$.  To proceed, we
need one further quantity that seems to be specific to $d=3$, and no
analogue in higher dimensions has been found so far. Let $M^{(ij)}$
denote the minors of $M\nts$, meaning the determinants of the
sub-matrices of $M$ that emerge from removing row $i$ and column
$j$. Then, define the real number
\[
    B_{_M} \, \defeq  \max_{1\leqslant i,j \leqslant 3}
    \frac{m^{}_{ii} m^{}_{jj}}{m^{}_{ij}} \ts
    (-1)^{i+j+\delta_{ij} -1} M^{(ij)}.
\] 
Now, we can recall the following results from the literature.

\begin{theorem}
   If\/ $M\in\cM_3$ is totally positive, one has the following results. 
\begin{enumerate}\itemsep=2pt   
\item If\/ $B_{_M} \geqslant \det(M) >0$, the matrix\/
   $M$ is \textnormal{g}-embeddable, and has a representation as a
   product of at most\/ $6$ Poisson matrices.
 \item If\/ $\det (M)>0$ with\/ $B_{_M} < \det(M)$, a representation as
   a product of at most\/ $6$ Poisson matrices exists if and only if a
   set of\/ $9$ inequalities and\/ $6$ equations is satisfied, as
   detailed in\/ \textnormal{\cite[Eq.~1.2]{Fryd83}}.
\item If\/ $M$ is \textnormal{g}-embeddable, it can always be written as 
   a product of finitely many Poisson matrices, the number of which is
   bounded by 
\[
  6 \left\lceil \frac{\ts \log (\det (M))}{\log (1/2)}
  \right\rceil .
\]      
\item If\/ $M$ is \textnormal{g}-embeddable with\/
  $\det (M) \geqslant \frac{1}{8}$, one has\/
  $B_{_M} \geqslant \det (M)$ and a representation as in\/ $(1)$, with
  at most\/ $6$ Poisson matrices.
\item If\/ $M$ is \textnormal{g}-embeddable with\/
  $B_{_M} < \det (M)$, one has a representation as a product of
  finitely many Poisson matrices. For\/ $k\geqslant 2$, their number
  is bounded by\/ $n^{}_{k}$, where
\[
    n^{}_{k} \, = \, \begin{cases}
      5k-2 \ts , & \text{if }\,  \tfrac{1}{2\cdot 8^{k-1}_{\vphantom{\chi}}}
                      \leqslant \det (M) 
                            < \tfrac{1}{8^{k-1}} \ts , \\
    5k-1 \ts , & \text{if } \, \tfrac{1^{\vphantom{I}}}{8^k} \leqslant 
                    \det (M)  < \tfrac{1}{2\cdot 8^{k-1}} \ts ,
    \end{cases}    
\]
   and it is known that\/ $6$ is generally not sufficient.
\end{enumerate}
\end{theorem}

\begin{proof}
  The first two claims follow from \cite[Thm.~1.1]{Fryd83}, but are
  also contained in the earlier work of \cite{Fryd80a,Fryd80b}.

  Claim (3) follows from \cite[Thm.~10]{JR}, while Claims (4) and (5)
  are a mild reformulation of \cite[Thm.~4.1]{Fryd83} and
  \cite[Thm.~4.2]{Fryd83}, respectively.
\end{proof}

While it is not known how good the representation bounds are with
growing $k$, it is clear that the number of matrices in a Bang-Bang
representation becomes more difficult to control as the determinant of
$M$ shrinks.

\section{Outlook}

The classic embedding problem for $d\leqslant 4$ is in fairly good
shape, and further classes of examples can be studied or refined with
the above tools and criteria. In this context, for $d=4$, the role of
algebraic models will become important, as is visible with symmetric
versus doubly stochastic generators; compare
Example~\ref{ex:doubly-stoch}. Various further classes for $d=4$ are
worth looking at, and a revised systematic treatment of embeddability
for time-reversible classes of Markov matrices seems useful, too.

While results in more than four dimensions are still limited, they are
nonetheless important, also in the context of population
genetics. Here, we only mention recombination, which can be understood
as a Markov process \cite{BB}. Clearly, embeddability is an important
feature again, and some progress will be possible because, in a
natural formulation with posets and lattices, one can work with
triangular Markov matrices, where the diagonal elements are the
eigenvalues, which then are non-negative real numbers.

Clearly, perhaps the most important next step will be the development
of the generalised embedding for $d=4$, as time-inhomogeneous Markov
chains are practically important and unavoidable. While we already
know that we cannot get all theoretically possible determinants this
way, relevant extensions do occur, though many questions remain
open. This is even true in the context of equal-input matrices, where
one needs to relate products with the natural grading, and to
understand precisely which additional cases emerge in this general
scheme.

\section*{Acknowledgements}

It is our pleasure to thank Ellen Baake and Tanja Schindler for
valuable discussions. We thank two anonymous referees for their
thoughtful comments, which helped us to improve the readability of
this paper. This work was supported by the German Research
Foundation (DFG, Deutsche Forschungsgemeinschaft) under \mbox{SFB
  1283/2 2021 -- 317210226}.

\end{document}